\pdfoutput=1
\documentclass{svproc}

\usepackage{url}
\usepackage{local-abbrv}
\usepackage{local-jmsdelim}
\usepackage{jms-sections}
\usepackage{preamble}

\usepackage{etoolbox}
\apptocmd{\sloppy}{\hbadness 10000\relax}{}{}

\begin{document}
\mainmatter

\title{Tensorial Structure of the Lifting Doctrine in Constructive Domain Theory}

\titlerunning{Tensorial Structure of the Lifting Doctrine}
\author{Jonathan Sterling}

\authorrunning{Jonathan Sterling}

\institute{%
  Computer Laboratory\\
  University of Cambridge\\
  \email{js2878@cl.cam.ac.uk}%
}

\maketitle

\begin{abstract}

  We present a survey of the two-dimensional and tensorial structure of the \emph{lifting doctrine} in constructive domain theory, \ie in the theory of directed-complete partial orders (dcpos) over an arbitrary elementary topos. We establish the universal property of lifting of dcpos as the Sierpi\'nski cone, from which we deduce (1) that lifting forms a Kock--Z\"oberlein doctrine, (2) that lifting algebras, pointed dcpos, and inductive partial orders form canonically equivalent locally posetal 2-categories, and (3) that the category of lifting algebras is cocomplete, with connected colimits created by the forgetful functor to dcpos. Finally we deduce the symmetric monoidal closure of the Eilenberg--Moore resolution of the lifting 2-monad by means of smash products; these are shown to classify both bilinear maps and strict maps, which we prove to coincide in the constructive setting. We provide several concrete computations of the smash product as dcpo coequalisers and lifting algebra coequalisers, and compare these with the more abstract results of Seal.
  Although all these results are well-known classically, the existing proofs do not apply in a constructive setting; indeed, the classical analysis of the Eilenberg--Moore category of the lifting monad relies on the fact that all lifting algebras are free, a condition that is not known to hold constructively.
\end{abstract}

\begin{xsect}{Introduction}

  Axiomatic approaches to domain theory take place in a \emph{monoidal adjunction} between a category of ``predomains'' and a category of ``domains''. The simplest notion of predomain is given by \emph{directed complete partial orders} (dcpos) and Scott--continuous functions between them; a corresponding notion of domain arises by considering algebras for an appropriate commutative monad on the preorder-enriched category of predomains. Most commonly, domains are considered to be algebras for a \emph{lifting monad} $\LL$ on predomains that introduces partiality.

  From this abstract definition, we may \emph{not} conclude that lifting is defined on points by taking the coproduct with $\One$, as Kock has pointed out~\cite{kock:1995}, unless the ambient topos is boolean; in general, we must use the partial map classifier of the ambient topos. This difference from classical domain theory is the source of many subtleties in the constructive setting.

  If the Eilenberg--Moore resolution $L\dashv U : \LiftAlg\to\Dcpo$ of the lifting monad is going to be monoidal, then $\LiftAlg$ would need to have a monoidal product $\otimes$; then the left adjoint being strong monoidal would mean that we have coherent isomorphisms $L\prn{A\times B}\cong LA\otimes LB$, \etc. Therefore we always know how to define the tensor product on \emph{free} domains, but it does not immediately follow from this that we may extend the tensor to operate on non-free domains. In classical mathematics, this difficulty is side-stepped by virtue of the fact that \emph{there are no non-free domains!}

  Indeed, classically, every $\LL$-algebra is a \emph{free} $\LL$-algebra --- if $X$ has a bottom element $\bot$, it can be seen that $X$ is the lift of the dcpo $X\setminus\brc{\bot}$ using the law of the excluded middle. Unfortunately, this simple description of $\LL$-algebras does not carry over to the constructive mathematics of an elementary topos, as Kock has discussed at length~\cite{kock:1995}. We can illustrate the problem by means of the following Brouwerian counterexample (\cref{thm:lem-vs-free-on-non-bottom}) which follows by way of \cref{lem:F-conservative} below --- anticipating a precise definition of lifting monad.

  \begin{proposition}\label{lem:F-conservative}
    The lifting functor $L\colon \Dcpo\to\LiftAlg$ is conservative.
  \end{proposition}

  \begin{proof}
    For any morphism of dcpos $f\colon A\to B$, the following is a pullback square:
    \begin{equation*}
      \DiagramSquare{
        nw = A,
        ne = B,
        sw = LA,
        se = LB,
        north = f,
        south = Lf,
        west = \eta_A,
        east = \eta_B,
        west/style = embedding,
        east/style = embedding,
        nw/style = pullback,
        width = 2.25cm,
        height = 1.5cm,
      }
    \end{equation*}

    Any pullback of an isomorphism is an isomorphism; therefore, if $Lf$ is an isomorphism, so is $f$.\qed
  \end{proof}

  \begin{theorem}\label{thm:lem-vs-free-on-non-bottom}
    The law of excluded middle holds if and only if every free $\LL$-algebra is free on its non-bottom elements.
  \end{theorem}

  \begin{proof}
    If the law of excluded middle holds, then obviously every $\LL$-algebra is free on its non-bottom elements. In the converse direction, we consider whether the $\LL$-algebra $\Omega$ given by the collection of all propositions with their implication order, where suprema are computed by existential quantification, is free on its non-bottom elements; it is easy to see that $\Omega$ is the free $\LL$-algebra on the terminal dcpo. Therefore the map ${L\prn{!\Sub{\Omega\setminus\brc{\bot}}}}\colon L\prn{\Omega\setminus\brc{\bot}}\to L\One$ is an isomorphism; by assumption, we may conclude from \cref{lem:F-conservative} that $\Omega\setminus\brc{\bot}$ is a singleton --- or, equivalently, that a proposition $\phi$ is true if and only if $\phi\not=\bot$.
    Now let $\psi$ be any proposition; to show that $\psi\lor\lnot\psi$, by the above we may assume $\lnot\prn{\psi\lor\lnot\psi}$ to prove a contradiction; our assumption is equivalent to $\lnot\psi \land \lnot\lnot\psi$, which is clearly contradictory.\qed
  \end{proof}

  Although \cref{thm:lem-vs-free-on-non-bottom} shows that it need not be the case that all $\LL$-algebras are free on their non-bottom elements, one might conjecture that every $\LL$-algebra is nonetheless free on a \emph{different} subdcpo. The most natural candidate for a subdcpo $X^+\subseteq UX$ such that $LX^+ \cong X$ would be the one spanned by \emph{positive} elements in the sense of De Jong and \'Escard\'o~\cite{dejong-escardo:2021} as adapted from Johnstone~\cite{johnstone:1984}: an element $x$ of an $\LL$-algebra $X$ is called \demph{positive} when any semidirected subset of $X$ whose suprema lies above $x$ is directed. Noting that the subposet of an $\LL$-algebra $X$ spanned by positive elements is always a dcpo, we are naturally led to the following open question:

  \begin{openquestion}\label{open-question:free-on-pos}
    Does there exist an elementary topos containing an $\LL$-algebra that is not free on its subdcpo of positive elements?
  \end{openquestion}

  Indeed, Kock has shown that an $\LL$-algebra is free if and only if it is free on its positive elements~\cite{kock:1995}; combining this with \cref{lem:F-conservative}, we see that the \emph{only} possible generators for a free $\LL$-algebra dcpo are its positive elements (which coincide with the non-bottom elements in the classical setting). Therefore, an answer to \cref{open-question:free-on-pos} would determine altogether whether and how all $\LL$-algebras can be free in constructive mathematics; I conjecture that the answer to \cref{open-question:free-on-pos} is ``Yes'', and so there may exist examples of non-free domains. Until and unless this expectation is contravened by mathematical evidence, the constructive version of the smash product must be defined on (potentially) non-free domains.

  \paragraph{Lifting closed structure \`a la Kock and Seal}
  \NewDocumentCommand\VCat{}{\mathcal{V}}

  It is a well-known result of category theory due to Kock~\cite{kock:1971} that the category of algebras $\VCat^{\mathbb{T}}$ for a commutative monad $\mathbb{T}\equiv\prn{T,\eta,\mu}$ on a symmetric monoidal closed category $\VCat$ with equalizers inherits \emph{closed} structure from $\VCat$, and (moreover) that the Eilenberg--Moore resolution of $\mathbb{T}$ consists of closed functors, \ie the left and right adjoints laxly preserve the internal hom. What is missing is the \emph{monoidal} structure on $\LL$-algebras that should extend the Eilenberg--Moore resolution $L\dashv U\colon \LiftAlg\to\Dcpo$ to a (symmetric) \emph{monoidal} closed adjunction. Luckily, a further result of Seal~\cite{seal:2013} provides sufficient conditions for a category of algebras to admit a tensor product by means of a construction dual to that of the internal hom and, moreover, for this tensor product to represent bilinear maps. That these conditions in fact hold constructively for dcpos and their lifting monad has not been verified until now, although they are not especially difficult.

  \paragraph{Summary of results}

  The contribution of the present paper is to provide a constructive analysis of the lifting doctrine for dcpos, embodied in the following results:
  \begin{enumerate}
    \item \textbf{Universal properties of $\Omega$:} the top truth value $\top\colon\One\hookrightarrow\Omega$ is the universal Scott--open immersion (\cref{lem:universal-open-immersion}), and the inequality $\bot\sqsubseteq\top \colon \One\hookrightarrow\Omega$ satisfies the 2-categorical universal property of the Sierpi\'nski space (\cref{lem:subobject-classifier-is-sigma}).\footnote{Although these results are known, they play a important role in what follows.}
    \item \textbf{Universal properties for lifting:} lifting enjoys both left- and right-handed universal properties in the 2-category of dcpos as a Sierpi\'nski cone (\cref{lem:lifting-is-scone}) and as a partial product (\cref{lem:lifting-is-partial-product}) respectively. The former implies our most important technical lemma, that $\bot\colon \One\hookrightarrow LA$ and $\eta_{A}\colon A\hookrightarrow LA$ are jointly (lax) epimorphic (\cref{cor:lax-epi}), enabling a restricted form of classical reasoning when establishing inequalities of the form $f\sqsubseteq g\colon LA\to B$.
    \item \textbf{Lifting is a Kock--Z\"oberlein doctrine:} for any lifting algebra $X$, the structure map $\alpha_X\colon LX\to X$ is left adjoint to the unit $\eta_X \colon X\hookrightarrow LX$, and so lifting algebra structures are unique (\cref{lem:lax-idempotent}).
    \item \textbf{Monadicity of pointed dcpos and ipos:} lifting algebras, pointed dcpos, and inductive partial orders are all canonically equivalent as locally posetal 2-categories (\cref{cor:lift-alg-characterisation}), and so pointed dcpos and ipos are monadic over dcpos (\cref{cor:monadicity}).
    \item \textbf{Cocompleteness of lifting algebras:} the category of lifting algebras is closed under all colimits, with connected colimits created by the forgetful functor $U\colon \LiftAlg\to\Dcpo$ (\cref{cor:U-creates-connected-colimits,cor:lift-algebras-cocomplete}).
    \item \textbf{Tensorial structure of lifting:} bilinear maps concide with bistrict maps (\cref{lem:bistrict-iff-bilinear}) and are representable by the \emph{smash product} (\cref{lem:uni-bistrict}) for which we provide several computations as coequalisers in both $\Dcpo$ and $\LiftAlg$ (\cref{cor:smash-coequalisers}). Smash products extend to a full symmetric monoidal structure on $\LiftAlg$, so that the adjunction $L\dashv U\colon\LiftAlg\to\Dcpo$ is symmetric monoidal (\cref{cor:lifting-adjunction-symmetric-monoidal}). Moreover, smash products are left adjoint to strict function spaces (\cref{lem:smash-lolli}) which make $L\dashv U$ into a \emph{closed} adjunction.
  \end{enumerate}

  \paragraph{Why does constructive domain theory matter?}

  The generality of our results is important, as modern approaches to programming semantics routinely involve computing recursive functions in non-boolean toposes. Our interest in constructive domains is not rooted in the philosophy of intuitionism, but instead in the practical necessity to study computation in \emph{variable and continuous sets}~\cite{lawvere:1975} as well as \emph{effective sets}~\cite{hyland:1982,bauer:2006}, whose dynamics generalize those of constant sets.

  In fact, it happens that the constructive theory of dcpos has not received much attention in the literature outside the groundbreaking work of Kock~\cite{kock:1995}, Townsend~\cite{townsend:1996}, and De Jong and \'Escard\'o~\cite{dejong-escardo:2021,dejong:2021,dejong:2023:thesis}. Therefore many results that appear to be ``obvious'' have not in fact been established, and the constructive domains behave differently enough from the classical ones that it would not be safe to take these results for granted. This paper is one further step in the direction of a thorough and base-independent account of dcpos that is applicable in an arbitrary topos.

  \paragraph{Topos-theoretic forerunners}
  Many of the results of the present paper have not previously been stated for dcpos, but their proofs nonetheless follow a well-trod template from locale theory and topos theory. For example, the two-dimensional analysis of lifting in terms of Sierpi\'nski cones and partial products was carried out for bounded toposes over a fixed elementary topos by Johnstone~\cite{johnstone:1992}\footnote{See also \S{}B4.4 of \emph{The Elephant}~\cite{johnstone:2002}.} and applied to the \emph{topical domain theory} of algebraic dpcos over a given topos by Vickers~\cite{vickers:1999}.
  On the other hand, not all dcpos come from a locale~\cite{johnstone:1981}: therefore, although our proofs are the ones that one naturally expects from experience with locales and toposes, the results must still be stated and proved for dcpos. With this said, we acknowledge that the two-dimensional analysis of dcpo lifting exposed here is well-known in the domain theoretic community and it is included in the present paper only for the sake of systematising existing knowledge.

  We are unsure if our main results concerning the cocompleteness of lifting algebras and their symmetric monoidal structure carry over to locales and toposes, but answering such a question would be a natural next step.

\end{xsect}

\begin{xsect}{Preliminaries in constructive category theory}
  \NewDocumentCommand\CCat{}{\mathcal{C}}
  \NewDocumentCommand\DCat{}{\mathcal{D}}
  \NewDocumentCommand\ICat{}{\mathcal{I}}
  \NewDocumentCommand\TT{}{\mathbb{T}}

  \begin{xsect}{Creation of colimits}
    In order to prove the cocompleteness of lifting algebras (\cref{sec:alg-cocomplete}), we will need some completely standard results about creation of colimits. Unfortunately, the categorical literature is saturated with subtly different and mutually incompatible definitions of what it means to create (co)limits. For example, Mac~Lane~\cite[Ch.~V]{maclane:1998} defines creation of (co)limits in a \emph{strict} way that involves equality of objects: as a result, it is not even the case that every equivalence of categories creates colimits. The non-invariance of Mac Lane's original notion is an actual impediment to practical use, as one naturally wishes to replace given categories by equivalent ones freely. For this reason, we adopt the following definition from Riehl~\cite[\S3.3]{riehl:2017}.

    \begin{definition}\label{def:created-colimit}
      Let $U\colon \DCat\to \CCat$ be a functor and let $\mathbb{D}$ be a class of diagrams in $\DCat$. The functor $U$ is said to \demph{create colimits of diagrams in $\mathbb{D}$} when for any diagram $D\colon \ICat\to\DCat$ in $\mathbb{D}$, if $UD\colon\ICat\to\CCat$ has a colimit then $D\colon \ICat\to\DCat$ has a colimit and $U$ both preserves and reflects colimits of $D$, \ie a cocone under $D$ is colimiting if and only if its image under $U$ is.
    \end{definition}

    The following are standard results of category theory, but we state and prove them carefully to avoid any doubt as to their constructivity or their compatibility with \cref{def:created-colimit}. Readers confident in the theory of created colimits would not miss much by skipping the remainder of this section.

    \begin{lemma}\label{lem:colimit:alg}
      Let $\CCat$ be a category and let $\TT \equiv \prn{T,\eta,\mu}$ be a monad on $\CCat$, and let $\mathbb{D}$ be a class of diagrams in $\CCat$. Suppose that the endofunctor $T$ preserves colimits of diagrams in $\mathbb{D}$. Let $X\colon \ICat\to\CCat^\TT$ be a diagram of $\TT$-algebras such that $UX\colon \ICat\to\CCat$ lies in $\mathbb{D}$ and has a universal cocone $c\colon UX\to\brc{C}$. We may extend $C$ to an essentially unique $\TT$-algebra structure $\bar{C}$ over $C$ in a canonical way such that $c\colon UX\to\brc{C}$ lifts to a cocone of algebras $\bar{c}\colon X\to \brc{\bar{C}}$ over $c$.
    \end{lemma}

    We will argue using the \emph{string diagrammatic} language of the 2-category of categories, the advantage being that it clarifies reasoning that involves naturality. We refer to Hinze and Marsden~\cite{hinze-marsden:2023} for a thorough introduction to string diagrams in a 2-category; note, however, that we differ from \opcit by having diagrams flow from the downward and to the right in keeping with the usual diagrammatic order of composition. In what follows, we let $F\dashv U$ be the Eilenberg--Moore resolution of $\TT$.

    \begin{proof}
      By assumption, the following diagram is a universal cocone.
      \begin{equation}\label[diagram]{diag:colimit:alg:a}

      \end{equation}

      We define a further cocone on the left below, which by the universal property of \cref{diag:colimit:alg:a} factors through a unique map $\beta\colon UFC\to C$ as depicted on the right:
      \begin{equation*}
        %
      \end{equation*}

      We will show that the map $\beta\colon UFC\to C$ satisfies the axioms of a $\TT$-algebra.
      \begin{enumerate}
        \item The unit law asserts that \cref{diag:colimit:alg:unit} below depicts the identity cell on $C$:
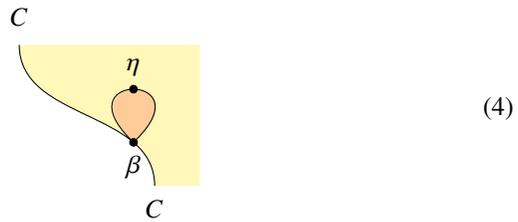
              \begin{equation}\label[diagram]{diag:colimit:alg:unit}
                %
              \end{equation}

              By the universal property of \cref{diag:colimit:alg:a}, it suffices to check that composition of \cref{diag:colimit:alg:a} with \cref{diag:colimit:alg:unit} is equal to \cref{diag:colimit:alg:a}. Forming the composite, we first recall the defining property of $\beta\colon UFC\to C$ and rewrite accordingly:
              \begin{equation*}
                %
              \end{equation*}

              Finally, we rewrite using the snake identity of $F\dashv U$:
              \begin{equation*}
                %
              \end{equation*}

        \item The multiplication law asserts that the following two diagrams are equal:
              \begin{equation*}
                %
              \end{equation*}

              It suffices to consider their restriction along the cocone $TTc\colon TTUX\to \brc{ TTC}$, which is universal as $T$ is assumed to preserve this colimit. We first use the defining property of $\beta$:
              \begin{equation*}
                %
              \end{equation*}

              We use the defining property of $\beta$ once more:
              \begin{equation*}
                %
              \end{equation*}

              Naturality allows us to swap the order in which the counits are composed, corresponding to the depth of the depicted ``sag''.
              \begin{equation*}
                %
              \end{equation*}

              Then the defining equation of $\beta$ implies the result.
              \begin{equation*}
                %
              \end{equation*}
      \end{enumerate}

      Hence we may define a $\TT$-algebra structure $\bar{C}$ with $U\bar{C}=C$, setting the structure map $\alpha_{\bar{C}} \colon TC\to C$ to be $\beta$. That $c$ lifts to a cocone of algebras is \emph{exactly} the defining condition of $\alpha_{\bar{C}}=\beta$ via the universal property of $Tc\colon TUX\to \brc{TC}$; uniqueness of the algebra structure follows from the same universal property.\qed
    \end{proof}

    \begin{lemma}\label{lem:colimit:reflect}
      Let $\TT\equiv\prn{T,\eta,\mu}$ be a monad on a category $\CCat$, and let $\mathbb{D}$ be a class of diagrams in $\CCat$. If $T$ preserves colimits of diagrams in $\mathbb{D}$, then $U\colon \CCat^\TT\to\CCat$ reflects colimits of diagrams in $\mathbb{D}$.
    \end{lemma}

    \begin{proof}
      Let $X\colon \ICat\to\CCat^\TT$ be a diagram equipped with a cocone $y\colon X\to \brc{Y}$ whose image $Uy\colon UX\to \brc{UY}$ in $\CCat$ is universal.
      \begin{equation*}

      \end{equation*}

      If therefore follows, by assumption, that $TUy\colon TUX\to\brc{TUY}$ is universal:
      \begin{equation}\label[diagram]{diag:colimit:reflect:TUy}
        %
      \end{equation}

      We aim to show that $y\colon X\to \brc{Y}$ is universal in $\CCat^\TT$. To check this universal property, we fix a further cocone $z\colon X \to \brc{Z}$ in $\CCat^\TT$; of course, we may factor $Uz\colon UX\to \brc{UZ}$ through the universal cocone $Uy\colon UX\to\brc{UY}$ through some unique $h\colon UY\to UZ$ as depicted below:
      \begin{equation}\label{diag:colimit:reflect:h-def}
        %
      \end{equation}

      We will show that $h\colon UY\to UZ$ lies in the image of some $\bar{h}\colon Y\to Z$ in $\CCat^\TT$; as the forgetful functor $U\colon \CCat^\TT\to\CCat$ is necessarily faithful, this will establish that $y\colon X\to\brc{Y}$ is a universal cocone. To exhibit $\bar{h}\colon Y\to Z$ over $h$ is, by definition, the same as to check that the latter is a homomorphism of algebras in the sense depicted below:
      \begin{equation}\label{diag:colimit:reflect:alg-law}
        %
      \end{equation}

      Because \cref{diag:colimit:reflect:TUy} is a universal cocone, we can check \cref{diag:colimit:reflect:alg-law} by restricting both sides along \cref{diag:colimit:reflect:TUy}. After doing so, we first use rewrite along \cref{diag:colimit:reflect:h-def}:
      \begin{equation*}
        %
      \end{equation*}

      We can then swap the order in which $z$ is composed with the counit, by naturality:
      \begin{equation*}
        %
      \end{equation*}

      We finally use \cref{diag:colimit:reflect:h-def} one last time.
      \begin{equation*}
        %
      \end{equation*}

      We have shown that $h\colon UY\to UZ$ satisfies the homomorphism property, and therefore lies in the image of some (unique) $\bar{h}\colon Y\to Z$, so we are done.\qed
    \end{proof}

    \begin{lemma}\label{lem:T-preserves-to-U-creates}
      Let $\TT\equiv\prn{T,\eta,\mu}$ be a monad on a category $\CCat$ and let $\mathbb{D}$ be a class of diagrams in $\CCat$. If $T$ preserves colimits of diagrams in $\mathbb{D}$, then $U\colon \CCat^\TT\to\CCat$ creates colimits of diagrams in $\mathbb{D}$.
    \end{lemma}

    \begin{proof}
      Let $X\colon \ICat\to \CCat^\TT$ be a diagram such that $UX\colon \ICat\to\CCat$ has a universal cocone $c\colon UX\to\brc{C}$ in $\CCat$. We let $\bar{C}\in\CCat^\TT$ with $U\bar{C}=C$ be the algebra structure on $C$ given by \cref{lem:colimit:alg}, so that $c\colon UX\to\brc{UC}$ lifts to a cocone of algebras $\bar{c}\colon X\to \brc{C}$. As $U\colon \CCat^\TT\to\CCat$ reflects colimits of diagrams in $\mathbb{D}$ (\cref{lem:colimit:reflect}), we conclude that the cocone $\bar{c}\colon X\to \brc{C}$ is indeed universal in $\CCat^\TT$.\qed
    \end{proof}
  \end{xsect}

\begin{xsect}{Geometry in a 2-category}
  \NewDocumentCommand\KK{}{\mathcal{K}}

  In this section, we elucidate the 2-categorical universal properties that will play a role in the constructive study of the lifting doctrine on dcpos. Although we of course have need only for \emph{poset-enriched} versions of what follows, we first work in as much generality as possible in order to lay the foundations for future investigations of higher-dimensional domain theory outside the locally posetal setting.

  \begin{definition}\label{def:sierpinski-space}
    Let $\KK$ be any 2-category with a terminal object; a \demph{Sierpi\'nski space} is then defined to be a cocomma object of the following form:
    \begin{equation*}
      \begin{tikzpicture}[diagram,baseline=(current bounding box.center)]
        \SpliceDiagramSquare{
          nw = \One,
          ne = \One,
          sw = \One,
          east = \top,
          south = \bot,
          se = \Sigma,
          north/style = {double},
          west/style = {double},
        }
        \node[between = sw and ne] {\rotatebox{45}{$\Rightarrow$}};
      \end{tikzpicture}
    \end{equation*}

    Equivalently, the Sierpi\'nski space is the \emph{tensor} $\Delta^1\cdot \One$ where $\Delta^1$ is the directed interval category $\brc{0\to 1}$.
  \end{definition}

  Reading \cref{def:sierpinski-space} in the 2-category of dcpos, the Sierpi\'nski space $\Sigma$ is, if it exists, the smallest dcpo that contains two points $\bot,\top:\Sigma$ and an inequality $\bot\sqsubseteq\top$.\footnote{As we will see, this description does \emph{not} imply that the Sierpi\'nski dcpo has exactly two points!} The Sierpi\'nski space is a special case of a more general gluing construction called the Sierpi\'nski cone:

  \begin{definition}\label{def:scone}
    The
    \demph{Sierpi\'nski cone} of an object $A:\KK$ in a 2-category $\KK$ with a terminal object is defined to be the following cocomma object:
    \begin{equation*}
      \begin{tikzpicture}[diagram,baseline=(current bounding box.center)]
        \SpliceDiagramSquare{
          nw = A,
          ne = A,
          sw = \One,
          east = \top,
          south = \bot,
          se = \Sigma A,
          west = {!}_A,
          north/style = {double},
        }
        \node[between = sw and ne] {\rotatebox{45}{$\Rightarrow$}};
      \end{tikzpicture}
    \end{equation*}
  \end{definition}

  The geometry of \cref{def:scone} is that $\Sigma{A}$ adjoins an additional point ``to the left'' of $A$, which forms the apex of a cone in $A$ whose endpoints lie in $A$. Of course, we have $\Sigma = \Sigma\One$ and further generic ``finite chain'' figures can be obtained by iteration; for instance, $\Sigma_n :\equiv \Sigma^n\One$ would be the generic chain with $n$ segments.

  \begin{observation}\label{lem:prod-preserves-cocomma}
    Product 2-functors $A\times-$ in a cartesian closed 2-category preserve cocomma squares.
  \end{observation}

  \begin{lemma}\label{lem:phoa}
    Let $\KK$ be a 2-category with a terminal object and an exponentiable Sierpi\'nski space $\Sigma$; then for any $Y\in\KK$, the following lax square induced by evaluation at the generic 2-cell $\bot\sqsubseteq\top$ is a comma square in $\KK$:
    \begin{equation}\label[diagram]{diag:lem:phoa}
      \begin{tikzpicture}[diagram,baseline=(current bounding box.center)]
        \SpliceDiagramSquare{
          nw = Y^\Sigma,
          ne = Y,
          sw = Y,
          north = -\top,
          west = -\bot,
          se = Y,
          south/style = double,
          east/style = double,
        }
        \node[between = sw and ne] {\rotatebox{45}{$\Rightarrow$}};
      \end{tikzpicture}
    \end{equation}
  \end{lemma}

  \begin{proof}
    Equivalently, we must check that $Y^\Sigma$ is the power $\Delta^1\pitchfork Y$. The proof is (2-)adjoint calisthenics, using the characterisation of $\Sigma$ as the power $\Delta^1\cdot\One$.
    \begin{align*}
      \KK\prn{X, Y^\Sigma}
       & \cong
      \KK\prn{X\times\Sigma, Y}
      \\
       & \cong \KK\prn{X\times \prn{\Delta^1\cdot\One},Y}
      \\
       & \cong \KK\prn{\Delta^1\cdot X, Y}
      \\
       & \cong \mathbf{Cat}\prn{\Delta^1,\KK\prn{X,Y}}
      \\
       & \cong \KK\prn{X,\Delta^1\pitchfork Y}
    \end{align*}

    Thus it follows that \cref{diag:lem:phoa} is a co-comma square.\qed
  \end{proof}
\end{xsect}
 
  \begin{xsect}{Partial products in a 2-category}
    \NewDocumentCommand\KK{}{\mathcal{K}}

    Finally, we recall the notion of \emph{(op)fibration} and \emph{partial product} in a 2-category~\cite[\S{}B4.4]{johnstone:2002}. In this section, let $\KK$ be a finitely complete 2-category. We will prefer the “Chevalley criterion” for opfibrations described below.

    \begin{definition}[Loregian and Riehl~\cite{loregian-riehl:2020}]\label{def:cocart-fib}
      A 1-cell $p\colon E\to B$ in $\KK$ is called an \demph{opfibration} when the canonical arrow $\Delta^1 \pitchfork E \to p\downarrow B$ corresponding to the lax square below has a left adjoint right inverse:
      \begin{equation*}
        \begin{tikzpicture}[diagram,baseline=(current bounding box.center)]
          \SpliceDiagramSquare{
            ne = B,
            se = B,
            sw = E,
            south = p,
            nw = \Delta^1\pitchfork E,
            east/style = double,
            west = \partial_0,
            north = p\circ\partial_1,
          }
          \node[between = sw and ne] {\rotatebox{45}{$\Rightarrow$}};
        \end{tikzpicture}
      \end{equation*}
    \end{definition}

    \begin{construction}[Lifting 2-cells to generalised fibers]\label{con:opfib-lift}
      As Hazratpour and Vickers~\cite{hazratpour-vickers:2020} point out, an opfibration in the sense of \cref{def:cocart-fib} can be equipped with operations corresponding to the more nuts-and-bolt description of internal opfibrations given by Johnstone~\cite{johnstone:2002}. In particular, for a given 2-cell $\alpha\colon f\to g$ in $\KK\prn{C,B}$ we may define a 1-cell $\alpha^*E\colon f^*E\to g^*E$ between pullbacks. In particular, the 2-cell determines a 1-cell $f^*E\to p\downarrow B$, where $p\circ p^*f \cong  f\circ f^*p$ is the canonical isomorphism of the pullback square:
      \begin{equation*}
        \begin{tikzpicture}[scale=0.75,baseline=(current bounding box.center)]
          \CreateRect{3}{3}
          \path
          coordinate[label=above:$\strut p^*f$] (n/p-f) at (spath cs:north 0.2)
          coordinate[label=below:$\strut f^*p$] (s/f-p) at (sw -| n/p-f)
          coordinate[label=above:$\strut p$] (n/p) at (spath cs:north 0.5)
          coordinate[label=below:$\strut g$] (s/g) at (spath cs:south 0.8)
          ;
          \begin{pgfinterruptboundingbox}
            \draw[spath/save=swooshr] (n/p-f) to[out=-90,in=90] (s/g);
            \draw[spath/save=swooshl] (n/p) to[out=-90,in=90] (s/f-p);
            \path[name intersections={of=swooshr and swooshl}] coordinate[dot,label=left:$\cong$] (phi) at (intersection-1);
            \path coordinate[dot,label=left:$\alpha$] (alpha) at (spath cs:swooshr 0.66);
            \begin{scope}[on background layer]
              \tikzset{
                spath/split at intersections={swooshl}{swooshr},
                spath/get components of={swooshr}\cmpr,
                spath/get components of={swooshl}\cmpl
              }
              \fill[cate] (nw) to (n/p-f) to[spath/use={\getComponentOf\cmpr1,weld}] (phi.center) to[spath/use={\getComponentOf\cmpl2,weld}] (s/f-p) to (sw) to cycle;
              \fill[catc] (n/p-f) to[spath/use={\getComponentOf\cmpr1,weld}] (phi.center) to[spath/use={\getComponentOf\cmpl1,weld,reverse}] (n/p) to cycle;
              \fill[catc] (s/f-p) to[spath/use={\getComponentOf\cmpl2,weld,reverse}] (phi.center) to[spath/use={\getComponentOf\cmpr2,weld}] (s/g) to cycle;
              \fill[catd] (ne) to (n/p) to[spath/use={\getComponentOf\cmpl1,weld}] (phi.center) to[spath/use={\getComponentOf\cmpr2,weld}] (s/g) to (se) to cycle;
            \end{scope}
          \end{pgfinterruptboundingbox}
        \end{tikzpicture}
      \end{equation*}

      Postcomposing with the left adjoint right inverse to $\Delta^1\pitchfork E \to p\downarrow B$, we obtain the following cells and equations:
      \begin{equation}\label[diagram]{diag:opfib-lift:g-bar}
        \begin{tikzpicture}[scale=0.75,baseline=(current bounding box.center)]
          \CreateRect{3}{2}
          \path
          coordinate[label=above:$\strut \bar g$] (n/g') at (spath cs:north 0.33)
          coordinate[label=above:$\strut p$] (n/p) at (spath cs:north 0.66)
          coordinate[label=below:$\strut f^*p$] (s/f-p) at (sw -| n/g')
          coordinate[label=below:$\strut g$] (s/g) at (sw -| n/p)
          ;
          \begin{pgfinterruptboundingbox}
            \draw[spath/save=swooshr] (n/g') to[out=-90,in=90] (s/g);
            \draw[spath/save=swooshl] (n/p) to[out=-90,in=90] (s/f-p);
            \path[name intersections={of=swooshr and swooshl}] coordinate[dot,label=left:$\cong$] (chi) at (intersection-1);
            \begin{scope}[on background layer]
              \tikzset{
                spath/split at intersections={swooshl}{swooshr},
                spath/get components of={swooshr}\cmpr,
                spath/get components of={swooshl}\cmpl
              }
              \fill[cate] (nw) to (n/g') to[spath/use={\getComponentOf\cmpr1,weld}] (chi.center) to[spath/use={\getComponentOf\cmpl2,weld}] (s/f-p) to (sw) to cycle;
              \fill[catc] (n/g') to[spath/use={\getComponentOf\cmpr1,weld}] (chi.center) to[spath/use={\getComponentOf\cmpl1,weld,reverse}] (n/p) to cycle;
              \fill[catc] (s/f-p) to[spath/use={\getComponentOf\cmpl2,weld,reverse}] (chi.center) to[spath/use={\getComponentOf\cmpr2,weld}] (s/g) to cycle;
              \fill[catd] (ne) to (n/p) to[spath/use={\getComponentOf\cmpl1,weld}] (chi.center) to[spath/use={\getComponentOf\cmpr2,weld}] (s/g) to (se) to cycle;
            \end{scope}
          \end{pgfinterruptboundingbox}
        \end{tikzpicture}
        \quad
        \begin{tikzpicture}[scale=0.75,baseline=(current bounding box.center)]
          \CreateRect{2}{2}
          \path
          coordinate[label=above:$\strut p^*f$] (n/p-f) at (spath cs:north 0.5)
          coordinate[label=below:$\strut \bar{g}$] (s/g') at (sw -| n/p-f);
          \begin{pgfinterruptboundingbox}
            \draw (n/p-f) to coordinate[dot,label=left:$\bar{\alpha}$] (alpha') (s/g');
            \begin{scope}[on background layer]
              \fill[cate] (nw) rectangle (s/g');
              \fill[catc] (n/p-f) rectangle (se);
            \end{scope}
          \end{pgfinterruptboundingbox}
        \end{tikzpicture}
      \end{equation}
      \begin{equation*}
        \begin{tikzpicture}[scale=0.75,baseline=(current bounding box.center)]
          \CreateRect{3}{3}
          \path
          coordinate[label=above:$\strut p^*f$] (n/p-f) at (spath cs:north 0.2)
          coordinate[label=below:$\strut f^*p$] (s/f-p) at (spath cs:south 0.5)
          coordinate[label=above:$\strut p$] (n/p) at (spath cs:north 0.8)
          coordinate[label=below:$\strut g$] (s/g) at (spath cs:south 0.8)
          ;
          \begin{pgfinterruptboundingbox}
            \draw[spath/save=swooshr] (n/p-f) to[out=-90,in=90] (s/g);
            \draw[spath/save=swooshl] (n/p) to[out=-90,in=90] (s/f-p);
            \path[name intersections={of=swooshr and swooshl}] coordinate[dot,label=left:$\cong$] (phi) at (intersection-1);
            \path coordinate[dot,label=left:$\bar\alpha$] (alpha') at (spath cs:swooshr 0.33);
            \begin{scope}[on background layer]
              \tikzset{
                spath/split at intersections={swooshl}{swooshr},
                spath/get components of={swooshr}\cmpr,
                spath/get components of={swooshl}\cmpl
              }
              \fill[cate] (nw) to (n/p-f) to[spath/use={\getComponentOf\cmpr1,weld}] (phi.center) to[spath/use={\getComponentOf\cmpl2,weld}] (s/f-p) to (sw) to cycle;
              \fill[catc] (n/p-f) to[spath/use={\getComponentOf\cmpr1,weld}] (phi.center) to[spath/use={\getComponentOf\cmpl1,weld,reverse}] (n/p) to cycle;
              \fill[catc] (s/f-p) to[spath/use={\getComponentOf\cmpl2,weld,reverse}] (phi.center) to[spath/use={\getComponentOf\cmpr2,weld}] (s/g) to cycle;
              \fill[catd] (ne) to (n/p) to[spath/use={\getComponentOf\cmpl1,weld}] (phi.center) to[spath/use={\getComponentOf\cmpr2,weld}] (s/g) to (se) to cycle;
            \end{scope}
          \end{pgfinterruptboundingbox}
        \end{tikzpicture}
        =
        \begin{tikzpicture}[scale=0.75,baseline=(current bounding box.center)]
          \CreateRect{3}{3}
          \path
          coordinate[label=above:$\strut p^*f$] (n/p-f) at (spath cs:north 0.2)
          coordinate[label=below:$\strut f^*p$] (s/f-p) at (sw -| n/p-f)
          coordinate[label=above:$\strut p$] (n/p) at (spath cs:north 0.5)
          coordinate[label=below:$\strut g$] (s/g) at (spath cs:south 0.8)
          ;
          \begin{pgfinterruptboundingbox}
            \draw[spath/save=swooshr] (n/p-f) to[out=-90,in=90] (s/g);
            \draw[spath/save=swooshl] (n/p) to[out=-90,in=90] (s/f-p);
            \path[name intersections={of=swooshr and swooshl}] coordinate[dot,label=left:$\cong$] (phi) at (intersection-1);
            \path coordinate[dot,label=left:$\alpha$] (alpha) at (spath cs:swooshr 0.66);
            \begin{scope}[on background layer]
              \tikzset{
                spath/split at intersections={swooshl}{swooshr},
                spath/get components of={swooshr}\cmpr,
                spath/get components of={swooshl}\cmpl
              }
              \fill[cate] (nw) to (n/p-f) to[spath/use={\getComponentOf\cmpr1,weld}] (phi.center) to[spath/use={\getComponentOf\cmpl2,weld}] (s/f-p) to (sw) to cycle;
              \fill[catc] (n/p-f) to[spath/use={\getComponentOf\cmpr1,weld}] (phi.center) to[spath/use={\getComponentOf\cmpl1,weld,reverse}] (n/p) to cycle;
              \fill[catc] (s/f-p) to[spath/use={\getComponentOf\cmpl2,weld,reverse}] (phi.center) to[spath/use={\getComponentOf\cmpr2,weld}] (s/g) to cycle;
              \fill[catd] (ne) to (n/p) to[spath/use={\getComponentOf\cmpl1,weld}] (phi.center) to[spath/use={\getComponentOf\cmpr2,weld}] (s/g) to (se) to cycle;
            \end{scope}
          \end{pgfinterruptboundingbox}
        \end{tikzpicture}
      \end{equation*}

      The isomorphism $p\circ \bar{g}\cong g\circ f^*p$ depicted in \cref{diag:opfib-lift:g-bar} is precisely the data of a suitable map $\alpha^*E\colon f^*E\to g^*E$, considering the universal property of $g^*E$.
    \end{construction}

    The following notion is described by Johnstone~\cite{johnstone:2002} as a \emph{partial product cone}.

    \begin{definition}[Johnstone~\cite{johnstone:2002}]
      Let $p\colon E\to B$ in be an opfibration in $\KK$, and let $A$ be a 0-cell in $\KK$. A \demph{nondeterministic map} from $C$ to $A$ with coefficients in $p\colon E\to B$ is defined to consist of a 1-cell $u\colon C\to B$ equipped with a further 1-cell $e\colon u^*E\to A$ as depicted below:
      \begin{equation*}
        \begin{tikzpicture}[diagram,baseline=(current bounding box.center)]
          \SpliceDiagramSquare<sq/>{
            ne = E,
            se = B,
            nw = u^*E,
            sw = C,
            south = u,
            east = p,
            nw/style = pullback,
          }
          \node[left = of sq/nw] (nw) {$A$};
          \draw[->] (sq/nw) to node[above] {$e$} (nw);
        \end{tikzpicture}
      \end{equation*}

      A morphism of such nondeterministic maps from $\prn{u,e}$ to $\prn{u',e'}$ is given by a 2-cell $\alpha\colon u\to u'$ together with a further 2-cell $\beta\colon e\to \alpha^*E \ast e'$ where $\alpha^*E\colon u^*E\to u'^*E$ is as described in \cref{con:opfib-lift}.
    \end{definition}

    We shall write $\KK^p\prn{-,A} \colon \OpCat{\KK}\to\mathbf{Cat}$ for the pseudofunctor sending 0-cells $C\in\KK$ to the category of nondeterministic maps from $C$ to $A$ with coefficients in $p$.

    \begin{definition}\label{def:partial-product}
      The \demph{partial product} of an opfibration $p\colon E\to B$ in $\KK$ with a 0-cell $A$ is a 0-cell $\mathcal{P}_{\bullet}\prn{p,A}$ representing the pseudofunctor $\KK^p\prn{-,A}$ in the sense that we have a pseudonatural equivalence $\KK\prn{-,\mathcal{P}_\bullet\prn{p,A}} \simeq \KK^p\prn{-,A}$.
    \end{definition}

    When $\mathcal{P}_\bullet\prn{p,A}$ is the partial product of $p\colon E\to B$ with $A\in\KK$, we have in the generic case a nondeterministic map from $\mathcal{P}_\bullet\prn{p,A}$ to itself, as depicted below:
    \begin{equation*}
      \begin{tikzpicture}[diagram,baseline=(current bounding box.center)]
        \SpliceDiagramSquare<sq/>{
          ne = E,
          se = B,
          nw = u^*E,
          sw = \mathcal{P}_\bullet\prn{p,A},
          south = u,
          east = p,
          nw/style = pullback,
        }
        \node[left = of sq/nw] (nw) {$A$};
        \draw[->] (sq/nw) to node[above] {$e$} (nw);
      \end{tikzpicture}
    \end{equation*}

    In this case, we shall refer to the above as the \demph{universal nondeterministic map out of $A$ with coefficients in $p\colon E\to B$.}

  \end{xsect}

\end{xsect}

\begin{xsect}{Basic notions in constructive domain theory}

  We recall the basics of the (constructive) theory of dcpos and their lifting monad, following the exposition of De~Jong and \'Escardo~\cite{dejong:2021,dejong-escardo:2021}. The main difference in relation to \opcit is that we assume propositional resizing, as we are not concerned here with predicativity.

  \begin{definition}
    A partial order $A$ is called a \demph{directed-complete} when any directed subset $U\subseteq A$ has a supremum in $A$. A morphism of directed-complete partial orders is a \demph{Scott--continuous function}, \ie a function that preserves directed suprema.
  \end{definition}

  We shall refer to directed-complete partial orders a \emph{dcpo}s, writing writing $\Dcpo$ for the category of dcpos and Scott continuous maps. Note that Scott--continuous functions are automatically monotone.

  \begin{xsect}{Open subspaces and their classifier}\label{sec:open-subsets}

    We recall the notion of Scott--open subset of a dcpo in the constructive setting, \eg from De~Jong~\cite{dejong:2021}.

    \begin{definition}
      A subset $U\subseteq A$ of a dcpo $A$ is called \demph{Scott--open} when it is upward closed and, moreover, inaccessible by directed suprema in the sense that for any directed subset $S\subseteq A$ with $\DLub{}{S} \in U$, there exists an element $s\in S$ such that $s\in U$.
    \end{definition}

    \begin{remark}
      Note that the appropriate notion of Scott--closed subset is \emph{not} obtained by taking complements of Scott--open subsets, except in the case of continuous dcpos~\cite{dejong:2021}. We will not deal with closed subsets in this paper.
    \end{remark}

    We shall refer to the subdcpo spanned by a given Scott--open subset as a \demph{Scott--open subspace}. A morphism of dcpos $i\colon A\to B$ factoring through an isomorphism onto an open subspace of $B$ is called a \demph{Scott--open immersion}.
    We will observe that \emph{universal monomorphism} $\top\colon \One\rightarrowtail\Omega$ in the category of sets\footnote{To be more precise, we mean the ambient topos when we speak of ``sets''.} extends to a \emph{universal Scott--open immersion} in the world of dcpos.

    \begin{lemma}\label{lem:subobject-classifier-is-dcpo}
      The universe $\Omega$ of propositions is a dcpo with its implication order.
    \end{lemma}
    \begin{proof}
      Implication clearly gives rise to a partial order on $\Omega$; the existential quantifier ensures that $\Omega$ is in fact a sup-lattice, and thus a dcpo.\qed
    \end{proof}

    \begin{theorem}\label{lem:universal-open-immersion}
      The morphism $\top\colon \One\hookrightarrow\Omega$ is the \demph{universal Scott--open immersion} in $\Dcpo$, in the sense that $\top\colon\One\hookrightarrow\Omega$ is a Scott--open immersion and that for any other Scott--open immersion $i\colon U\hookrightarrow A$, there exists a unique cartesian square from $i$ to $\top$ in $\Dcpo$ as depicted below:
      \begin{equation*}
        \DiagramSquare{
          ne = \One,
          se = \Omega,
          sw = A,
          nw = U,
          west = i,
          east = \top,
          north = !_U,
          nw/style = pullback,
          south = \exists! \brk{i},
          west/style = embedding,
          east/style = embedding,
          south/style = {exists,->},
          north/style = {exists,->},
          height = 1.5cm,
        }
      \end{equation*}
    \end{theorem}

    \begin{proof}
      Without loss of generality, we may consider the open immersion induced by a Scott--open subset $U$ of $A$. As the forgetful functor from dcpos to their underlying sets is faithful, we can deduce our result from the universal property of $\top\colon \One\rightarrowtail\Omega$ as the universal monomorphism in the category of sets; in particular, it is enough to observe that the characteristic function of a subset of a dcpo is Scott--continuous if and only if the subset is Scott--open, recalling that joins in $\Omega$ are given by existential quantification.\qed
    \end{proof}

  \end{xsect}

  \begin{xsect}{Geometry of the Scott--open subspace classifier}\label{sec:sierpinski-dcpo}
    $\Dcpo$ is easily seen to be enriched in posets; given $f,g\colon A\to B$ we define $f\sqsubseteq g$ if and only if $f x \sqsubseteq g x$ for all $x:A$. This enrichment turns $\Dcpo$ into a (locally posetal) 2-categories, and so we may consider 2-categorical limits and colimits.

    We have seen a ``right-handed'' or limit-style universal property for $\Omega$ as the base of the universal Scott--open immersion (\cref{lem:universal-open-immersion}). In this section, we will see that $\Omega$ has an alternative left-handed universal property as the \emph{Sierpi\'nski space} (\cref{def:sierpinski-space}) in the 2-category of dcpos. These two universal properties reflect the role of $\Omega$ as a dualising object in the algebro-geometric context of domain theory.

    \begin{theorem}\label{lem:subobject-classifier-is-sigma}
      The following is a cocomma square in the 2-category $\Dcpo$, and so $\Omega$ is the Sierpi\'nski space in the sense of \cref{def:sierpinski-space}:
      \begin{equation*}
        \begin{tikzpicture}[diagram,baseline=(current bounding box.center)]
          \SpliceDiagramSquare{
            nw = \One,
            ne = \One,
            sw = \One,
            east = \top,
            south = \bot,
            se = \Omega,
            north/style = {double},
            west/style = {double},
          }
          \node[between = sw and ne] {\rotatebox{45}{$\Rightarrow$}};
        \end{tikzpicture}
      \end{equation*}
    \end{theorem}

    \begin{proof}
      Consider an arbitrary lax square in the following configuration:
      \begin{equation*}
        \begin{tikzpicture}[diagram,baseline=(current bounding box.center)]
          \SpliceDiagramSquare{
            nw = \One,
            ne = \One,
            sw = \One,
            east = c_1,
            south = c_0,
            se = C,
            north/style = {double},
            west/style = {double},
          }
          \node[between = sw and ne] {\rotatebox{45}{$\Rightarrow$}};
        \end{tikzpicture}
      \end{equation*}

      The universal map $h\colon \Omega\to C$ factoring $c_0$ through $\bot$ and $c_1$ through $\top$ is defined so as to send $\phi:\Omega$ to $\Lub{\One+\phi}\brk{c_0\mid c_1}$, \ie supremum of the union of $\Compr{c_1}{\phi = \top}$ and $\brc{c_0}$. It is also observed easily that this assignment preserves directed suprema in $\Omega$. That  $h\colon \Omega\to C$ is unique with this factorization property follows from the uniqueness of suprema: any map factoring $c_0$ and $c_1$ in this sense is supremum of the same directed subset.\qed
    \end{proof}

    By virtue of \cref{lem:subobject-classifier-is-sigma}, we may define $\Sigma :\equiv \Omega$; therefore, unless the law of excluded middle holds, it need not be the case that $\Sigma$ has exactly two points --- although the law of non-contradiction ensures that no third point can be proved unequal to both $\bot$ and $\top$.\footnote{From the external point of view, there will generally be many distinct \emph{global} points of the internal dcpo $\Sigma$. But even if $p,q,r$ are distinct global points, the topos logic will not deduce $p\not=q\not=r$ unless the topos is the \emph{empty topos} (\ie the trivial category).}

    \begin{remark}
      It is perhaps surprising at first that the Sierpi\'nski space in the 2-category of \emph{posets} nonetheless has only two elements in constructive mathematics, in spite of \cref{lem:subobject-classifier-is-sigma}. This is not so strange, however: the ideal completion 2-functor from posets to dcpos is \emph{left adjoint} to the forgetful functor, and so it necessarily preserves Sierpi\'nski objects. But in constructive mathematics, the set of ideals in $\Two = \brc{0\leq 1}$ necessarily contains \emph{all} directed downsets of $\Two$ and not just the decidable ones; thus we see, by means of a more conceptual argument than that of \cref{lem:subobject-classifier-is-dcpo}, that the Sierpi\'nski dcpo must be given by $\Omega$.
    \end{remark}

    \begin{lemma}\label{lem:top-is-cocart-fib}
      The universal open immersion $\top\colon\One\hookrightarrow\Sigma$ is a opfibration of dcpos in the sense of \cref{def:cocart-fib}.
    \end{lemma}

    \begin{proof}
      Letting $A$ be an arbitrary dcpo; we must check that the canonical morphism $\Delta^1\pitchfork \One \to \top\downarrow \Sigma$ has a left adjoint right inverse. In fact, $\Delta^1\pitchfork\One\cong \One\cong\top\downarrow\Sigma$, so we are done.\qed
    \end{proof}

    \begin{definition}[Paths between dpco morphisms]
      Let $f,g \colon A \to B$ be a morphism of dcpos; a \emph{path} from $f$ to $g$ is defined to be a morphism $\alpha\colon \Sigma\times A\to B$ such that $\alpha\circ\prn{\bot,-} = f$ and $\alpha\circ\prn{\top,-} = g$.
    \end{definition}

    \begin{corollary}[Path enrichment]\label{cor:path-enrichment}
      The following properties of paths hold:
      \begin{enumerate}
        \item There is at most one path between any two morphisms $f,g\colon A \to B$ of dcpos.
        \item For $f,g\colon A \to B$, there exists a path from $f$ to $g$ if and only if $f\sqsubseteq g$.
      \end{enumerate}
    \end{corollary}

    \begin{proof}
      These are immediate consequences of \cref{lem:subobject-classifier-is-sigma}.\qed
    \end{proof}

  \end{xsect}

\begin{xsect}{Enriched cocompleteness of the category of dcpos}

  Our study of the Sierpi\'nski space and the path-enrichment of $\Dcpo$ (\cref{cor:path-enrichment}) implies the important property that any 1-categorical colimits of dcpos that we may construct are, in fact, 2-categorical colimits.

  \begin{corollary}[Enrichment of colimits]\label{lem:dcpo-colimits-enriched}
    Colimits in $\Dcpo$ are poset-enriched.
  \end{corollary}

  \begin{proof}
    This follows immediately from \cref{cor:path-enrichment} and the fact that the product functor $\Sigma\times-$ has a right adjoint and is therefore cocontinuous.
  \end{proof}

  We have not seen, however, how to actually construct any given colimit of dcpos. Although it is not hard to see that $\Dcpo$ is cocomplete in a classical metatheory using the adjoint functor theorem~\cite{abramsky-jung:1996}, it is unclear how to satisfy the solution set condition in constructive mathematics, although it may nonetheless be possible. Luckily, it happens that the constructive cocompleteness of $\Dcpo$ is an immediate consequence of the (fully constructive) generalized coverage theorem of Townsend~\cite{townsend:1996}.

  \begin{lemma}[{Townsend~\cite[p.~72]{townsend:1996}}]\label{lem:dcpo-cocomplete}
    The category of dcpos is closed under coequalisers, and is therefore cocomplete.
  \end{lemma}

  \begin{proof}
    Townsend~\cite{townsend:1996} has shown that the coequaliser of dcpos can be computed in their enveloping sup-lattices and then extracted by means of an image factorization that isolates the smallest subdcpo of the coequaliser sup-lattice containing the original dcpo that we wished to quotient.\qed
  \end{proof}

  The argument of \opcit is a more conceptual version of the explicit construction of dcpo quotients in terms of dcpo presentations~\cite{jung-moshier-vickers:2008}, or the even more explicit constructions of Fiech~\cite{fiech:1996} and Goubault-Larrecq~\cite{goubault-larrecq:blog:dcpo-colimits}.

\end{xsect}
 
\end{xsect}
\begin{xsect}{The lifting monad and its algebras}

  \begin{xsect}{The partial map classifier}

    In this section, we shall study the structure of \emph{partial maps} of dcpos in terms of 2-category--theoretic universal properties.

    \begin{definition}
      A \demph{partial map} from $A$ to $B$ is given by a span $A\hookleftarrow U\rightarrow B$ in which $U\hookrightarrow A$ is a Scott--open immersion. An \emph{inequality} from $A\hookleftarrow U\rightarrow B$ to $A\hookleftarrow U'\rightarrow B$ is given by an embedding $U\hookrightarrow U'$ making both triangles below commute:
      
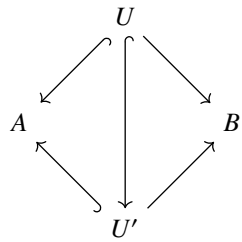
\[
        \begin{tikzpicture}[diagram]
          \node (n) {$U$};
          \node (w) [below left = of n] {$A$};
          \node (e) [below right = of n] {$B$};
          \node (s) [below right = of w] {$U'$};
          \draw[embedding] (n) to (w);
          \draw[embedding] (s) to (w);
          \draw[->] (n) to (e);
          \draw[->] (s) to (e);
          \draw[embedding] (n) to (s);
        \end{tikzpicture}
      \]
    \end{definition}

    \begin{observation}
      The partial order of partial maps from $A$ to $B$ is \emph{precisely} the (posetal) category $\Dcpo^\top\prn{A,B}$ of nondeterministic maps from $A$ to $B$ with coefficients in the universal Scott--open immersion $\top\colon \One\hookrightarrow\Sigma$.
    \end{observation}

    \begin{proof}
      This follows immediately from the universal property of $\top\colon\One\hookrightarrow\Sigma$ as the universal Scott--open immersion (\cref{lem:top-is-cocart-fib}).\qed
    \end{proof}

    Thus the appropriate enriched / 2-categorical universal property for \emph{classifying} partial maps is given by partial products (\cref{def:partial-product}). We will now give an explicit description of the classification of partial maps into $B$ by the partial product $\mathcal{P}_\bullet\prn{\top,B}$.

    \begin{construction}[The lifting operation on dcpos]\label{def:lifting}
      For a dcpo $A$, the \demph{lifted dcpo} $LA$ has the base of the \emph{partial map classifier} $LA :\equiv \Sum{\phi:\Omega}{A^\phi}$ as its underlying set, with the following partial order:
      \begin{align*}
        \prn{\phi,u} \sqsubseteq_{LA} \prn{\psi,v}
         & \Longleftrightarrow
        \Forall{x:\phi}
        \Exists{y:\psi}
        ux \sqsubseteq vy
        \\
         & \Longleftrightarrow
        \prn{\phi\sqsubseteq_\Sigma \psi}
        \land
        \Forall{x:\phi, y:\psi}
        ux \sqsubseteq vy
      \end{align*}

      If we write $\eta\colon A\hookrightarrow LA$ for the unit map sending $a$ to $\prn{\top,\Lam{x}{a}}$, then we see that we also have the following logical equivalence:
      \[
        \prn{\phi,u} \sqsubseteq_{LA} \prn{\psi,v}
        \Longleftrightarrow
        \Forall{a:A}
        a\in \eta\Inv u
        \to
        \Exists{b : B}
        b \in \eta\Inv v
      \]

      It is not difficult to show that if $A$ is directed-complete, then so is $LA$; suprema are computed so that the (clearly monotone) projection $\pi_1\colon LA \to \Sigma$ is Scott-continuous and so a morphism of dcpos.
    \end{construction}

    \begin{theorem}\label{lem:lifting-is-partial-product}
      Each lifted dcpo $LB$ is the partial product of $\top\colon\One\hookrightarrow\Sigma$ with $B$.
    \end{theorem}

    \begin{proof}
      We must construct an isomorphism of posets from $\Dcpo\prn{A,LB}$ to the poset $\Dcpo^\top\prn{A,B}$ of partial maps from $A$ to $B$. Given $f\colon A\to LB$, we choose the following partial map from $A$ to $B$:
      \begin{equation*}
        \begin{tikzpicture}[diagram,baseline=(current bounding box.center)]
          \SpliceDiagramSquare<sq/>{
            ne = \One,
            se = \Sigma,
            nw = \Compr{x:A}{\pi\prn{fx}=\top},
            sw = A,
            south = \pi\circ f,
            east = \top,
            east/style = embedding,
            west/style = embedding,
            nw/style = pullback,
            width = 2.5cm,
          }
          \node[left = 3cm of sq/nw] (nw) {$B$};
          \draw[->] (sq/nw) to node[above] {$\pi_2\circ f$} (nw);
        \end{tikzpicture}
      \end{equation*}

      Monotonicity is immediate. Conversely, we consider an arbitrary partial map:
      \begin{equation*}
        \begin{tikzpicture}[diagram,baseline=(current bounding box.center)]
          \SpliceDiagramSquare<sq/>{
            ne = \One,
            se = \Sigma,
            nw = U,
            sw = A,
            south = p,
            east = \top,
            east/style = embedding,
            west/style = embedding,
            nw/style = pullback,
          }
          \node[left = of sq/nw] (nw) {$B$};
          \draw[->] (sq/nw) to node[above] {$e$} (nw);
        \end{tikzpicture}
      \end{equation*}

      The above corresponds to the map $A\to LB$ sending $x:A$ to $\prn{px, \lambda z\mathpunct{.} \prn{x,z}}$.\qed
    \end{proof}

    \begin{corollary}
      Let $A$ be a dcpo; then the evaluation map $e\colon U\to A$ in the universal nondeterministic map with coefficients in $\top\colon\One\hookrightarrow\Sigma$ is an isomorphism.
      \begin{equation*}
        \begin{tikzpicture}[diagram,baseline=(current bounding box.center)]
          \SpliceDiagramSquare<sq/>{
            ne = \One,
            se = \Sigma,
            nw = U,
            sw = \mathcal{P}_\bullet\prn{\top,A},
            south = \pi,
            east = \top,
            east/style = embedding,
            west/style = embedding,
            nw/style = pullback,
          }
          \node[left = of sq/nw] (nw) {$A$};
          \draw[->] (sq/nw) to node[above] {$e$} (nw);
        \end{tikzpicture}
      \end{equation*}
    \end{corollary}

  \end{xsect}

  \begin{xsect}{Geometry of the partial map classifier}

    In \cref{sec:open-subsets,sec:sierpinski-dcpo} we have seen that the classifier of Scott--open subsets has an additional left-handed universal property as a 2-categorical colimit: the Sierpi\'nski space. In this section, we will upgrade this result to see that the partial map classifier of a given dcpo $A$ has an additional left-handed universal property as the Sierpi\'nski \emph{cone} of $A$. From this, we will obtain the most important reasoning principle for the lifting doctrine in constructive domain theory, namely our \cref{cor:joint-epi,cor:lax-epi}.

    \begin{theorem}[Lifting = Sierpi\'nski cone]\label{lem:lifting-is-scone}
      For any dcpo $A$, the following lax square involving the lifting operation is a co-comma square:
      \begin{equation*}
        \begin{tikzpicture}[diagram,baseline=(current bounding box.center)]
          \SpliceDiagramSquare{
            nw = A,
            ne = A,
            sw = \One,
            east = \eta_A,
            south = \bot,
            se = LA,
            west = {!}_A,
            east/style = embedding,
            south/style = embedding,
            north/style = {double},
          }
          \node[between = sw and ne] {\rotatebox{45}{$\Rightarrow$}};
        \end{tikzpicture}
      \end{equation*}

      In other words, the lifted dcpo $LA$ is in fact the \emph{Sierpi\'nski cone} of $A$ in $\Dcpo$.
    \end{theorem}

    \begin{proof}
      Consider an arbitrary lax square in the following configuration:
      \begin{equation*}
        \begin{tikzpicture}[diagram,baseline=(current bounding box.center)]
          \SpliceDiagramSquare{
            nw = A,
            ne = A,
            sw = \One,
            east = c_1,
            south = c_0,
            se = C,
            north/style = {double},
          }
          \node[between = sw and ne] {\rotatebox{45}{$\Rightarrow$}};
        \end{tikzpicture}
      \end{equation*}

      The universal map $h\colon LA\to C$ factoring $c_0$ through $\bot$ and $c_1$ through $\eta_A$ is defined so as to send $u : LA$ to the supremum of the union of $\brc{c_0}$ with $\Compr{c_1 x}{u = \eta_A x}$. This set is evidently directed, and so each $hu$ is well-defined; to see that the assignment $u\mapsto hu$ is continuous, we fix a directed subset $V\subseteq LA$:
      \begin{align*}
        h\DLub{}V
         & =
        \DLub{}\prn{
          \brc{c_0}\cup \Compr{c_1 x}{\DLub{}V=\eta_A x}
        }
        \\
         & =
        \DLub{}\prn{
          \brc{c_0}\cup \Compr{c_1 x}{\eta_A x\in V}
        }
        \\
         & =
        \DLub{u\in V}\prn{
          \brc{c_0}\cup \Compr{c_1 x}{u=\eta_A x}
        }
        \\
         & =
        \DLub{u\in V} hu
      \end{align*}

      Lastly, we must check that $h\colon LA\to C$ is unique with this property. We will show that any two $h,h'\colon LA\to C$ factoring our lax square in the appropriate sense are equal, fixing $u: LA$.
      \begin{align*}
        hu
         & =
        h \DLub{}\prn{
          \brc{\bot} \cup
          \Compr{\eta_A x}{u = \eta_A x}
        }
        \\
         & =
        \DLub{}\prn{
          \brc{h\bot}
          \cup
          \Compr{h\prn{\eta_A x}}{u = \eta_A x}
        }
        \\
         & =
        \DLub{}\prn{
          \brc{h'\bot}
          \cup
          \Compr{h'\prn{\eta_A x}}{u=\eta_A x}
        }
        \\
         & =
        h' \DLub{}\prn{
          \brc{\bot} \cup
          \Compr{\eta_A x}{u = \eta_A x}
        }
        \\
         & = h'u
      \end{align*}
      Thus $h = h'$.
      \qed
    \end{proof}

    From the universal property of $LA$ as the Sierpi\'nski cone of $A$, we can deduce the following important reasoning principle.

    \begin{corollary}\label{cor:joint-epi}
      For any dcpo $A$, the two embeddings $\bot \colon \One\hookrightarrow LA$ and $\eta_A\colon A \hookrightarrow LA$ are jointly epimorphic; as such, we have an epimorphic embedding $\brk{\bot \mid \eta_A} \colon \One+A\hooktwoheadrightarrow LA$.
    \end{corollary}

    \begin{proof}
      This is an immediate consequence of \cref{lem:lifting-is-scone}: because $LA$ is the Sierpi\'nski cone of $A$, equality of maps $LA\to C$ can be checked by restriction along the embeddings $\bot \colon \One\hookrightarrow LA$ and $\eta_A\colon A \hookrightarrow LA$.\qed
    \end{proof}

    Obviously, the category of dcpos is not balanced or else we would have $\One+A\cong LA$.
    It was Fiore~\cite{fiore:1995} who first argued for the importance of \cref{cor:joint-epi} for the general axiomatics of lifting monads as Kock--Z\"oberlein doctrines, \ie \emph{lax idempotent 2-monads}. In this paper, we consider a stronger \emph{enriched} version of this statement.

    \begin{corollary}\label{cor:lax-epi}
      For any dcpo $A$, the embedding $\brk{\bot \mid \eta_A} \colon \One+A\hooktwoheadrightarrow LA$ is \emph{lax epimorphic} in the 2-category of dcpos, so that for any dcpo $C$ the induced restriction map
      $\Dcpo\prn{\brk{\bot \mid \eta_A},C}\colon \Dcpo\prn{LA,C}\to \Dcpo\prn{\One+A,C}$ is an \emph{order-embedding}.
    \end{corollary}

    \begin{proof}
      This is a consequence of \cref{lem:dcpo-colimits-enriched,cor:joint-epi}.\qed
    \end{proof}

  \end{xsect}

  \begin{xsect}{Lifting as a 2-monad}

    It is not difficult to see that the lifting operation on dcpos is functorial and, indeed, a monad; on point-sets, these operations are the same as those of the (discrete) partial map classifier on sets --- as the functorial action sends continuous maps to continuous maps, and both the unit and multiplication can be seen to be continuous. Moreover, the functorial action is in fact \emph{monotone} in hom posets. Therefore:

    \begin{lemma}[Enrichment]
      Lifting gives rise to a 2-monad $\LL = \prn{L,\eta,\mu}$ on $\Dcpo$.
    \end{lemma}

    \begin{proof}
      This amounts to the fact that each functorial map taking $f\colon A \to B$ to $Lf\colon LA \to LB$ is \emph{monotone} as a function on hom posets. That the unit and multiplication are 2-natural is automatic in the locally posetal setting.\qed
    \end{proof}

    Essentially by definition, the Kleisli 2-category for $\LL$ is given by dcpos with \emph{partial} maps between them. The rest of this section is devoted to understanding the broader Eilenberg--Moore resolution of $\LL$, which extends beyond the free lifting algebras to arbitrary lifting algebras. We will show in \cref{sec:lift-alg-characterisation} that lifting algebras, pointed dcpos, and inductive partial orders give equivalent 2-categories; in \cref{sec:alg-cocomplete}, we will show that the category of lifting algebras is cocomplete.

    \begin{definition}
      We shall emphasise the property of dcpo morphism $f\colon UX \to YU$ tracking a morphism of $\LL$-algebras by calling it \demph{linear}.
    \end{definition}

    The following can be seen by unfolding definitions.
    \begin{observation}
      Each unit map $\eta_A \colon A\to LA$ is an order-embedding.
    \end{observation}

  \end{xsect}

  \begin{xsect}{Lifting as a Kock--Z\"oberlein doctrine}

    The lifting 2-monad is \emph{lax idempotent} and so gives rise to a Kock--Z\"oberlein doctrine on dcpos. We will see this doctrine takes the form of cocompletion under bottom elements, constructivising the classical viewpoint of dcpo lifting algebras as \emph{pointed} dcpos.

    \begin{lemma}\label{lem:lax-idempotent}
      The lifting 2-monad is lax idempotent:
      for any algebra $X\in \LiftAlg$, the structure map $\alpha_X \colon LUX \to UX$ is left adjoint to the unit $\eta_{UX} \colon UX \to LUX$ in $\Dcpo$.
    \end{lemma}

    \begin{proof}
      The counit $\alpha_X\circ \eta_{UX} \sqsubseteq \Idn{UX}$ is automatic (and invertible) by the unit law for monad algebras. To exhibit the unit $\Idn{LUX}\sqsubseteq \eta_{UX}\circ \alpha_X$, it suffices by \cref{cor:lax-epi} to check both $\bot\sqsubseteq \eta_{UX} \alpha_X\bot$ and $\eta_{UX}\sqsubseteq \eta_{UX} \alpha_X \eta_{UX}$. The former is immediate and the latter holds by the unit law for monad algebras.\qed
    \end{proof}

    \begin{corollary}\label{cor:algebras-unique}
      There is at most one lifting algebra structure on a dcpo.
    \end{corollary}

    \begin{proof}
      Left adjoints are unique!\qed
    \end{proof}

  \end{xsect}

  \begin{xsect}[sec:lift-alg-characterisation]{Lifting algebras, pointed dcpos, and ipos}
    The abstract notion of a lifting algebra can be identified with two more concrete notions: pointed dcpos and inductive partial orders (ipos).

    \begin{definition}
      A subset $U\subseteq A$ of a partial order $A$ is called \demph{semidirected} when for any two $x,y\in U$ there exists an upper bound for $x$ and $y$ in $U$. A subset is called \demph{directed} when it is both semidirected and inhabited.
    \end{definition}

    \begin{definition}
      A partial order $A$ is called \demph{inductive} when any semidirected subset $U\subseteq A$ has a supremum in $A$. A morphism of inductive partial orders is an \demph{inductive function}, \ie one that preserves semidirected suprema.
    \end{definition}

    We shall abbreviate inductive partial orders as \emph{ipos}, writing $\Ipo$ for the category of ipos and morphisms of ipos.

    \begin{definition}
      A dpco $A$ is called \demph{pointed} when it has a bottom element $\bot$, \ie such that $\bot\sqsubseteq a$ for all $a:A$.
    \end{definition}

    \begin{definition}
      A Scott--continuous map between pointed dcpos is called \demph{strict} when it preserves the bottom element.
    \end{definition}

    We shall abbreviate pointed dcpos as \emph{dcppos} and write $\Dcppo$ for the category of pointed dcpos and strict maps.

    \begin{lemma}\label{lem:pointed-vs-semidirected-complete}
      A dcpo $A$ is pointed if and only if it is inductive, i.e. semidirected-complete.
    \end{lemma}

    \begin{proof}
      Suppose that $A$ is closed under suprema of semidirected subsets. Then the supremum of the \emph{empty} subset (which is trivially semidirected) is can be seen to be the bottom element using the universal property of suprema.

      Conversely, suppose that $A$ is pointed and let $I\subseteq A$ be semidirected. Then we may replace $I\subseteq A$ by the \emph{directed} subset $I' = I\cup\brc{\bot}$; the inclusion $I\subseteq I'$ is clearly cofinal as $\bot$ lies below everything, so the supremum of $I'$ is also the supremum of $I$.\qed
    \end{proof}

    \begin{lemma}\label{lem:strict-vs-inductive}
      A Scott--continuous morphism between pointed dcpos is strict if and only if it is inductive, i.e. preserves suprema of semidirected subsets.
    \end{lemma}

    \begin{proof}
      An inductive morphism obviously preserves the bottom element. Conversely, let $f\colon A \to B$ preserve directed suprema and the bottom element and let $I\subseteq A$ be a semidirected subset of $A$. To show that $f\DLub{}{I} = \DLub{i:I}f i$, we note that $I\subseteq I\cup\brc{\bot}$ is a cofinal inclusion onto a \emph{directed} subset, and so $f\DLub{}{I} = f\DLub{}\prn{\brc{\bot}\cup I} = \DLub{\One+I}\brk{f\bot \mid f} = \DLub{\One+I}\brk{\bot\mid f} = \DLub{i:I}{f i}$.\qed
    \end{proof}

    \begin{lemma}[Pointed dcpos are lifting algebras]\label{lem:point-to-alg}
      Any pointed dcpo carries a lifting algebra structure.
    \end{lemma}

    Of course, by \cref{cor:algebras-unique} any lifting algebra structure we impose on a dcpo, pointed or not, is unique.

    \begin{proof}
      Let $A$ be a pointed dcpo; we define the structure map $\alpha_A\colon LA\to A$ to take $u:LA$ to the supremum of the semidirected subset $\Compr{x:A}{u=\eta_A x}$, computed via \cref{lem:pointed-vs-semidirected-complete}. The unit law is trivial, and the multiplication law follows from the fact that a supremum of suprema can be computed as the supremum of a single subset.\qed
    \end{proof}

    \begin{lemma}[Lifting algebras are pointed]\label{lem:alg-to-point}
      For any lifting algebra $X\in \LiftAlg$, the underlying dcpo $UX$ is pointed.
    \end{lemma}

    \begin{proof}
      The bottom element of $UX$ is obtained by applying the structure map to the bottom element of $LUX$, so we have $\bot :\equiv \alpha_X\prn{\bot,*}$. That this does in fact compute the bottom element can be seen as follows: fixing $u: UX$, we note that $\alpha_X\prn{\bot,*} \sqsubseteq_{UX} u$ is equivalent to $\bot\sqsubseteq_{LUX} \eta_{UX}u$ because $\alpha_X\dashv\eta_{UX}$ by \cref{lem:lax-idempotent} (lax idempotence).\qed
    \end{proof}

    \begin{lemma}[Strict maps \vs homomorphisms]\label{lem:strict-vs-homomorphism}
      A Scott--continuous map between pointed dcpos is strict if and only if it tracks a lifting algebra homomorphism.
    \end{lemma}

    \begin{proof}
      It is clear from the proof of \cref{lem:alg-to-point} that a homomorphism of algebras must preserve the bottom element. On the other hand, we suppose that $f\colon A \to B$ is strict to check that the following diagram commutes:
      \begin{equation}\label[diagram]{diag:strict-vs-homomorphism}
        \DiagramSquare{
          nw = LA,
          ne = LB,
          sw = A,
          se = B,
          west = \alpha_A,
          east = \alpha_B,
          north = Lf,
          south = f,
        }
      \end{equation}

      By \cref{cor:joint-epi} and the fact that all maps in sight are strict, it is enough to consider the restriction of \cref{diag:strict-vs-homomorphism} along $\eta_A\colon A\hookrightarrow LA$; then we have $\alpha_B\circ Lf\circ \eta_A = \alpha_B\circ \eta_A\circ f = f = f\circ\alpha_A\circ\eta_A$ by the unit law for algebras.\qed
    \end{proof}

    \begin{corollary}\label{cor:lift-alg-characterisation}
      The 2-categories of lifting algebras, pointed dcpos, and inductive partial orders are all canonically equivalent.
    \end{corollary}
    \begin{proof}
      Having and preserving bottom elements, semidirected suprema, and lifting algebra structures are all \emph{properties} (we have seen the latter in \cref{cor:algebras-unique}). Therefore, we may argue that these categories all arise as the same (non-full) subcategory of $\Dcpo$ via \cref{lem:point-to-alg,lem:alg-to-point,lem:strict-vs-homomorphism,lem:pointed-vs-semidirected-complete,lem:strict-vs-inductive}.\qed
    \end{proof}

    \begin{corollary}[Monadicity]\label{cor:monadicity}
      The forgetful functors $\Dcppo\to \Dcpo$ and $\Ipo\to\Dcpo$ are both monadic.
    \end{corollary}

  \end{xsect}

\begin{xsect}{Cocompleteness of lift-algebras}\label{sec:alg-cocomplete}

  \NewDocumentCommand\ICat{}{\mathcal{I}}

  \begin{lemma}\label{lem:lifting-preserves-connected-colimits}
    The lifting endofunctor $L\colon \Dcpo\to \Dcpo$ preserves connected colimits.
  \end{lemma}

  \begin{proof}
    Let $A_\bullet\colon \ICat\to\Dcpo$ be a connected diagram, \ie such that $\ICat$ is inhabited and has a finite zig-zag between any two objects; further suppose that there exists a universal cocone ${a_\bullet}\colon A_\bullet\to\brc{A_\infty}$, to check that the lifted cocone $La_\bullet\colon LA_\bullet\to \brc{LA_\infty}$ is also universal. We fix a cocone $b_\bullet\colon LA_\bullet\to\brc{B}$ and must check that there exists a unique map $b_\infty\colon LA_\infty \to B$ factoring $b_\bullet$ through $La_\bullet$.
    We have shown in \cref{lem:lifting-is-scone} that $LA_\infty$ is the Sierpi\'nski cone of $A_\infty$, so a map $b_\infty \colon LA_\infty\to B$ is uniquely determined by an element $b_\infty^\bot\colon \One \to B$ and a map ${b_\infty^\top}\colon A_\infty\to B$ such that $b_\infty^\bot\circ {!}_{A_\infty} \sqsubseteq b_\infty^\top$.

    We first define $b_\infty^\bot$ to be the unique element of $B$ that is equal to $b_k\bot$ for all $k\in\ICat$; that this element is exists and is unique follows from connectedness of $\ICat$. Next, we define $b_\infty^\top\colon A_\infty\to B$ using the universal property of ${a_\bullet}\colon A_\bullet\to \brc{A_\infty}$:
    \begin{equation*}
      \DiagramSquare{
        nw = A_\bullet,
        ne = \brc{A_\infty},
        sw = LA_\bullet,
        se = \brc{B},
        north = a_\bullet,
        east = \brc{b_\infty^\top},
        east/style = {exists, ->},
        west = \eta_{A_\bullet},
        west/style = open immersion,
        south = b_\bullet,
      }
    \end{equation*}

    Finally we check that $b_\infty^\bot\circ {!}_{A_\infty} \sqsubseteq b_\infty^\top$; by \cref{lem:dcpo-colimits-enriched}, it suffices to check that $b_\infty^\bot\circ {!}_{A_i} \sqsubseteq b_\infty^\top\circ a_i$ for each $i\in\ICat$; fixing $x : A_i$, we do indeed have $b_\infty^\bot = b_i\bot \sqsubseteq b_i\prn{\eta_{A_i} x} = b_\infty^\top\prn{a_i x}$ by monotonicity of $b_i\colon LA_i\to B$ on $\bot\sqsubseteq \eta_{A_i}x$.

    Thus we have defined a map $b_\infty\colon LA_\infty\to B$ such that $b_\infty \bot = b_k\bot$ for all $k\in\ICat$ and $b_\infty\prn{\eta_{A_\infty}x} = b_\infty^\top x$ for all $x:A_\infty$. We need to check that $b_\infty\colon LA_\infty\to B$ uniquely factors ${b_\bullet}\colon LA_\bullet\to \brc{B}$ through $LA_\bullet \colon LA_\bullet\to \brc{LA_\infty}$:
    \begin{equation*}
      \begin{tikzpicture}[diagram,baseline=(current bounding box.center)]
        \node (nw) {$LA_\bullet$};
        \node[right = of nw] (ne) {$\brc{LA_\infty}$};
        \node[below = of ne] (se) {$\brc{B}$};
        \draw[->] (nw) to node[above] {$La_\bullet$} (ne);
        \draw[->] (ne) to node[right] {$b_\infty$} (se);
        \draw[->] (nw) to node[sloped,below] {$b_\bullet$} (se);
      \end{tikzpicture}
    \end{equation*}

    We check the factorization above using \cref{cor:joint-epi}. In particular, it is enough to check that $b_\infty\prn{La_i\prn{\bot}} = b_i\bot$ and that $b_\infty\prn{La_i\prn{\eta_{LA_i}x}} = b_i\prn{\eta_{A_i}x}$ for each $x:A_i$. The former holds as we have $b_\infty\prn{La_i\prn{\bot}} = b_\infty\bot = b_\infty^\bot = b_i\bot$, and the latter holds by $b_\infty\prn{La_i\prn{\eta_{LA_i}x}} = b_\infty\prn{\eta_{LA_\infty}\prn{a_i x}} = \beta_\infty^\top\prn{a_i x} = b_i\prn{\eta_{A_i}x}$.
    Finally, we check that any two factorizations $f,g\colon LA_\infty\to B$ of $b_\bullet$ through $LA_\bullet$ are equal. But this follows by construction via \cref{cor:joint-epi} and the universal property of the cocone $a_\bullet\colon A_\bullet\to\brc{A_\infty}$.\qed
  \end{proof}

  \begin{corollary}\label{cor:U-creates-connected-colimits}
    The category of lift-algebras is closed under connected colimits, and these are created by the forgetful functor $U\colon \LiftAlg\to\Dcpo$.
  \end{corollary}

  \begin{proof}
    By \cref{lem:dcpo-cocomplete,lem:lifting-preserves-connected-colimits,lem:T-preserves-to-U-creates}.\qed
  \end{proof}

  \begin{lemma}[{Linton~\cite{linton:1969:coequalizers}}]\label{lem:lift-algebras-coproducts}
    The category of lift-algebras is closed under coproducts.
  \end{lemma}

  \begin{proof}
    Coproducts in $\LiftAlg$ are computed using a reflexive coequaliser involving the coproducts from $\Dcpo$. By \cref{cor:U-creates-connected-colimits}, we know that $\LiftAlg$ is closed under reflexive coequalisers and these are computed as in $\Dcpo$.\qed
  \end{proof}

  \begin{corollary}\label{cor:lift-algebras-cocomplete}
    The category of lift-algebras is cocomplete.
  \end{corollary}

  \begin{proof}
    By \cref{cor:U-creates-connected-colimits,lem:lift-algebras-coproducts}.\qed
  \end{proof}

\end{xsect}
 
\end{xsect}
\begin{xsect}{Tensorial structure of the lifting adjunction}

  \begin{xsect}{Enrichment and commutativity of the lifting monad}

    We shall view $\Dcpo$ as a symmetric monoidal closed category via its cartesian product and exponential, canonically self-enriched. We first observe that $\LiftAlg = \Dcppo = \Ipo$ inherits this $\Dcpo$-enrichment.

    \begin{lemma}
      The category $\Dcppo$ of pointed dcpos is $\Dcpo$-enriched in the sense that every hom poset $\Dcppo\prn{A,B}$ is closed under suprema of directed subsets.
    \end{lemma}

    \begin{proof}
      Given pointed dcpos $A$ and $B$, we must check that the supremum of a directed set of strict maps from $A$ to $B$, computed in the dcpo exponential $B^A$, is strict. This holds because function application is continuous, so we have
      $
        \prn{\DLub{i:I}f_i}\bot = \DLub{i:I}f_i\bot = \DLub{i:I}\bot = \bot
      $.\qed
    \end{proof}

    \begin{lemma}\label{lem:dcpo-powers}
      The category $\Dcppo$ of pointed dcpos is closed under $\Dcpo$-\emph{powers}.
    \end{lemma}

    \begin{proof}
      Let $A$ be a dcpo and let $B$ be a pointed dcpo. The \emph{power} $A\pitchfork B$ of $B$ by $A$ has the dcpo exponential $B^A$ as its underlying (pointed) dcpo.
      To check the universal property, we observe that a strict map from $C$ to $A\pitchfork B$ is the same as a map from $C\times A$ to $B$ that is strict in its first argument. Of course, this is the same as a Scott--continuous map from $A$ to $\Dcppo\prn{C,B}$. Thus we have $\Dcppo\prn{C,A\pitchfork B} \cong \Dcpo\prn{A,\Dcppo\prn{C,B}}$ and so we are done.\qed
    \end{proof}

    \begin{lemma}\label{lem:lifting-dcpo-enriched}
      The poset-enrichment of the lifting monad $\LL$ on $\Dcpo$ extends to a $\Dcpo$-enrichment.
    \end{lemma}

    \begin{proof}
      The functorial action and monadic operations can all be internalised as Scott--continuous functions.\qed
    \end{proof}

    \begin{corollary}\label{lem:lifting-strong}
      The lifting monad $\LL$ extends to a \emph{strong} monad on $\Dcpo$.
    \end{corollary}

    \begin{proof}
      Strengths for a given monad on a cartesian closed category $\mathcal{V}$ correspond precisely to $\mathcal{V}$-enrichments of the monad~\cite{mcdermott-uustalu:2022}.\qed
    \end{proof}

    \begin{lemma}
      The $\Dcpo$-enriched lifting monad $\LL$ is commutative.
    \end{lemma}

    \begin{proof}
      We use Kock's criterion for commutativity of a strong monad on a closed category. Fixing a pointed dcpo $B$ and a dcpo $A$, we must check that the extension map $\prn{-}^\dagger\colon A\pitchfork B\to LA\pitchfork B$ is \emph{strict}. As the bottom element of any power $I\pitchfork B$ is pointwise, we are trying to check that $\prn{\Lam{x}\bot}^\dagger u = \bot$ for any $u:LA$. By \cref{cor:joint-epi}, it suffices to observe that $\prn{\Lam{x}\bot}^\dagger\bot = \bot$ and $\prn{\Lam{x}\bot}^\dagger\prn{\eta_Aa} = \prn{\Lam{x}\bot}\prn{a} = \bot$.\qed
    \end{proof}

    \begin{corollary}
      The lifting monad $\LL$ is symmetric monoidal.
    \end{corollary}

    \begin{proof}
      This is in fact equivalent to being commutative.\qed
    \end{proof}

    \begin{construction}[Commutator]
      We define the commutator \[ \kappa_{A,B}\colon LA\times LB \to L\prn{A\times B}\] by iterated (internal) Kleisli extension; the commutativity property ensures that it doesn't matter in which order these extensions are taken.
    \end{construction}

  \end{xsect}

  \begin{xsect}[sec:smash-product]{Smash products and the universal bistrict morphism}

    \begin{lemma}\label{lem:tfae-pre-bistrictness}
      The following are equivalent for a morphism of dcpos $f\colon A\times B\to C$ where $A$ and $B$ are pointed:
      \begin{enumerate}
        \item Any of the following diagrams commute:
              \begin{equation*}
                \label[diagram]{diag:pre-bistrict:pure-pure}
                \begin{tikzpicture}[diagram,baseline=(0.base)]
                  \node(0) {$A+B$};
                  \node[right = 4.25cm of 0] (1) {$A\times B$};
                  \node[right = of 1] (2) {$C$};
                  \draw[->,transform canvas={yshift=.1cm}] (0) to node[above] {$\bot\circ\Bang{A+B}$} (1);
                  \draw[->,transform canvas={yshift=-.1cm}] (0) to node[below] {$\brk{\prn{\Idn{A},\bot} \mid \prn{\bot,\Idn{B}}}$} (1);
                  \draw[->] (1) to node[above] {$f$} (2);
                \end{tikzpicture}%
              \end{equation*}
              \begin{equation}\label[diagram]{diag:pre-bistrict:lift-pure}
                \begin{tikzpicture}[diagram,baseline=(0.base)]
                  \node(0) {$L\prn{A+B}$};
                  \node[right = 4.5cm of 0] (1) {$A\times B$};
                  \node[right = of 1] (2) {$C$};
                  \draw[->,transform canvas={yshift=.1cm}] (0) to node[above] {$\bot\circ\Bang{L\prn{A+B}}$} (1);
                  \draw[->,transform canvas={yshift=-.1cm}] (0) to node[below] {$\brk{\prn{\Idn{A},\bot} \mid \prn{\bot,\Idn{B}}}^\dagger$} (1);
                  \draw[->] (1) to node[above] {$f$} (2);
                \end{tikzpicture}%
              \end{equation}

        \item For any $a:A$ and $b:B$ we have $f\prn{\bot,b}= f\prn{a,\bot}$.
      \end{enumerate}

      \begin{proof}
        The last condition is immediately equivalent to \cref{diag:pre-bistrict:pure-pure} commuting. The equivalence between \cref{diag:pre-bistrict:pure-pure,diag:pre-bistrict:lift-pure} is deduced from \cref{cor:joint-epi}, noting that the parallel maps in \cref{diag:pre-bistrict:lift-pure} are both strict.\qed
      \end{proof}
    \end{lemma}

    \begin{lemma}\label{lem:tfae-pre-bistrictness:pointed-codomain}
      The following are equivalent for a (not necessarily strict) morphism $f\colon A\times B\to C$ of dcpos where $A$, $B$, and $C$ are pointed:
      \begin{enumerate}
        \item Any of the equivalent conditions of \cref{lem:tfae-pre-bistrictness}.
        \item Either of the following diagrams commute:
              \begin{equation}\label[diagram]{diag:pre-bistrict:pure-lift}
                \begin{tikzpicture}[diagram,baseline=(0.base)]
                  \node(0) {$A+B$};
                  \node[right = 5cm of 0] (1) {$L\prn{A\times B}$};
                  \node[right = of 1] (2) {$C$};
                  \draw[->,transform canvas={yshift=.1cm}] (0) to node[above] {$\bot\circ\Bang{A+B}$} (1);
                  \draw[->,transform canvas={yshift=-.1cm}] (0) to node[below] {$\eta_{A\times B}\circ \brk{\prn{\Idn{A},\bot} \mid \prn{\bot,\Idn{B}}}$} (1);
                  \draw[->] (1) to node[above] {$f^\dagger$} (2);
                \end{tikzpicture}
              \end{equation}
              \begin{equation}\label[diagram]{diag:pre-bistrict:lift-lift}
                \begin{tikzpicture}[diagram,baseline=(0.base)]
                  \node(0) {$L\prn{A+B}$};
                  \node[right = 5cm of 0] (1) {$L\prn{A\times B}$};
                  \node[right = of 1] (2) {$C$};
                  \draw[->,transform canvas={yshift=.1cm}] (0) to node[above] {$\bot\circ\Bang{L\prn{A+B}}$} (1);
                  \draw[->,transform canvas={yshift=-.1cm}] (0) to node[below] {$L\brk{\prn{\Idn{A},\bot} \mid \prn{\bot,\Idn{B}}}$} (1);
                  \draw[->] (1) to node[above] {$f^\dagger$} (2);
                \end{tikzpicture}
              \end{equation}
      \end{enumerate}
    \end{lemma}
    \begin{proof}
      \cref{diag:pre-bistrict:pure-pure} commutes if and only if \cref{diag:pre-bistrict:pure-lift} commutes, by the unit law for $C$ as a lifting algebra; for the same reason, \cref{diag:pre-bistrict:lift-pure} commutes if and only if \cref{diag:pre-bistrict:lift-lift} commutes.\qed
    \end{proof}

    \begin{definition}[Bistrict morphism]\label{def:bistrict}
      Let $A$, $B$, and $C$ be pointed dcpos. A Scott--continuous morphism $f\colon A\times B \to C$ is called \demph{bistrict} when any of the following equivalent conditions hold:
      \begin{enumerate}
        \item The morphism $f\colon A\times B\to C$ is strict and satisfies any of the equivalent conditions of \cref{lem:tfae-pre-bistrictness,lem:tfae-pre-bistrictness:pointed-codomain}.
        \item For any $a:A$ and $b:B$ we have $f\prn{\bot,b}= f\prn{a,\bot} = \bot$.
      \end{enumerate}
    \end{definition}

    \begin{theorem}[The universal bistrict map]\label{lem:uni-bistrict}
      For any pointed dcpos $A$ and $B$, we may define a pointed $A\otimes B$ equipped with a \emph{universal} bistrict map $\otimes_{A,B}\colon A\times B\to A\otimes B$, in the sense that any bistrict $f\colon A\times B\to C$ factors uniquely through it by a unique strict map $\bar{f}\colon A\otimes B\to C$ as depicted below:
      \begin{equation}\label[diagram]{diag:uni-bistrict}
        \begin{tikzpicture}[diagram, baseline=(current bounding box.center)]
          \node (nw) {$A\times B$};
          \node[right = of nw] (ne) {$C$};
          \node[below = of nw] (sw) {$A\otimes B$};
          \draw[->] (nw) to node[left] {$\otimes_{A,B}$} (sw);
          \draw[->] (nw) to node[above] {$f$} (ne);
          \draw[->,exists] (sw) to node[sloped,below] {$\exists! \bar{f}$} (ne);
        \end{tikzpicture}
      \end{equation}

      Moreover, the following diagram is a coequaliser in $\Dcppo$:
      \begin{equation}\label[diagram]{diag:smash-coeq:dcppo}
        \begin{tikzpicture}[diagram,baseline=(0.base)]
          \node(0) {$L\prn{A+B}$};
          \node[right = 4.5cm of 0] (1) {$A\times B$};
          \node[right = 2.5cm of 1] (2) {$A\otimes B$};
          \draw[->,transform canvas={yshift=.1cm}] (0) to node[above] {$\bot\circ\Bang{L\prn{A+B}}$} (1);
          \draw[->,transform canvas={yshift=-.1cm}] (0) to node[below] {$\brk{\prn{\Idn{A},\bot} \mid \prn{\bot,\Idn{B}}}^\dagger$} (1);
          \draw[->>] (1) to node[above] {$\otimes_{A,B}$} (2);
        \end{tikzpicture}
      \end{equation}
    \end{theorem}

    \begin{proof}
      We may compute the desired coequaliser, as we have already shown in \cref{cor:lift-algebras-cocomplete} that $\LiftAlg=\Dcppo$ is cocomplete. The coequaliser map $\otimes_{A,B}\colon A\times B\twoheadrightarrow A\otimes B$ is bistrict by definition, as \cref{diag:smash-coeq:dcppo} is an instance of \cref{diag:pre-bistrict:lift-pure} from \cref{lem:tfae-pre-bistrictness}. The unique factorisation condition of \cref{diag:uni-bistrict} is, then, precisely the universal property of \cref{diag:smash-coeq:dcppo} as a coequaliser in $\Dcppo$.\qed
    \end{proof}

    \begin{corollary}\label{cor:smash-coequalisers}
      The following are coequaliser diagrams in both $\Dcppo$ and $\Dcpo$:
      \begin{equation*}
        \begin{tikzpicture}[diagram,baseline=(0.base)]
          \node(0) {$L\prn{A+B}$};
          \node[right = 4.5cm of 0] (1) {$A\times B$};
          \node[right = 2.5cm of 1] (2) {$A\otimes B$};
          \draw[->,transform canvas={yshift=.1cm}] (0) to node[above] {$\bot\circ\Bang{L\prn{A+B}}$} (1);
          \draw[->,transform canvas={yshift=-.1cm}] (0) to node[below] {$\brk{\prn{\Idn{A},\bot} \mid \prn{\bot,\Idn{B}}}^\dagger$} (1);
          \draw[->>] (1) to node[above] {$\otimes_{A,B}$} (2);
        \end{tikzpicture}
        \tag{\ref{diag:smash-coeq:dcppo}}
      \end{equation*}
      \begin{equation}\label[diagram]{diag:smash-coeq:2}
        \begin{tikzpicture}[diagram,baseline=(0.base)]
          \node(0) {$L\prn{A+B}$};
          \node[right = 4.5cm of 0] (1) {$L\prn{A\times B}$};
          \node[right = 2.75cm of 1] (2) {$A\otimes B$};
          \draw[->,transform canvas={yshift=.1cm}] (0) to node[above] {$\bot\circ\Bang{L\prn{A+B}}$} (1);
          \draw[->,transform canvas={yshift=-.1cm}] (0) to node[below] {$L\brk{\prn{\Idn{A},\bot} \mid \prn{\bot,\Idn{B}}}$} (1);
          \draw[->>] (1) to node[above] {$\otimes_{A,B}^\dagger$} (2);
        \end{tikzpicture}
      \end{equation}
      \medskip

      The following are coequaliser diagrams in $\Dcpo$:
      \begin{equation}\label[diagram]{diag:smash-coeq:3}
        \begin{tikzpicture}[diagram,baseline=(0.base)]
          \node(0) {$A+B$};
          \node[right = 4.5cm of 0] (1) {$A\times B$};
          \node[right = 2.5cm of 1] (2) {$A\otimes B$};
          \draw[->,transform canvas={yshift=.1cm}] (0) to node[above] {$\bot\circ\Bang{A+B}$} (1);
          \draw[->,transform canvas={yshift=-.1cm}] (0) to node[below] {$\brk{\prn{\Idn{A},\bot} \mid \prn{\bot,\Idn{B}}}$} (1);
          \draw[->>] (1) to node[above] {$\otimes_{A,B}$} (2);
        \end{tikzpicture}
      \end{equation}
      \begin{equation}\label[diagram]{diag:smash-coeq:4}
        \begin{tikzpicture}[diagram,baseline=(0.base)]
          \node(0) {$A+B$};
          \node[right = 4.5cm of 0] (1) {$L\prn{A\times B}$};
          \node[right = 2.75cm of 1] (2) {$A\otimes B$};
          \draw[->,transform canvas={yshift=.1cm}] (0) to node[above] {$\bot\circ\Bang{A+B}$} (1);
          \draw[->,transform canvas={yshift=-.1cm}] (0) to node[below] {$\brk{\prn{\Idn{A},\bot} \mid \prn{\bot,\Idn{B}}}$} (1);
          \draw[->>] (1) to node[above] {$\otimes_{A,B}^\dagger$} (2);
        \end{tikzpicture}
      \end{equation}
    \end{corollary}

    \begin{proof}
      We have seen in \cref{cor:U-creates-connected-colimits} that the forgetful functor $U\colon \Dcppo\to\Dcpo$ creates connected colimits; therefore, a coequaliser diagram $\Dcppo$ is equally well a coequaliser diagram in $\Dcpo$. \cref{diag:smash-coeq:dcppo} is therefore a coequaliser in both categories by \cref{lem:uni-bistrict}. That \cref{diag:smash-coeq:2,diag:smash-coeq:3,diag:smash-coeq:4} are all coequalisers follows from \cref{lem:tfae-pre-bistrictness:pointed-codomain}.\qed
    \end{proof}

    \begin{lemma}\label{lem:commutator-is-universal}
      Up to isomorphism, the lifting monad sends any cartesian product $A\times B$ to the smash product $LA\times LB$. In particular, the commutator $\kappa_{A,B}\colon LA\times LB \to L\prn{A\times B}$ is the universal bistrict map in the sense of \cref{lem:uni-bistrict}.
    \end{lemma}

    \begin{proof}
      It suffices to show that any bistrict map $f\colon LA\times LB \to C$ extends uniquely along $\kappa_{A,B}\colon LA\times LB\to L\prn{A,B}$. We let $\bar{f}\colon L\prn{A\times B}\to C$ be the extension of $f\circ \eta_A\times \eta_B\colon A\times B\to C$, which is automatically strict. Uniqueness of the extension is deduced using \cref{cor:joint-epi}.\qed
    \end{proof}
  \end{xsect}

  \begin{xsect}{Bilinear morphisms and Seal's general theory}
    Although we have developed smash products and their universal property (\cref{sec:smash-product}) with respect to bistrict morphisms in the concrete, Seal~\cite{seal:2013} has provided a general theory for deducing tensorial structure from commutative monads. In this section, we show that \opcit's notion of \emph{bilinear map} coincides with our bistrict maps and, moreover, that the tensor products of \opcit satisfy the same universal property as our smash product.

    \begin{definition}[Bilinear morphism~\cite{seal:2013}]\label{def:bilinear}
      Let $A$, $B$, and $C$ be pointed dcpos. A Scott--continuous morphism $f\colon A\times B \to C$ is called \demph{bilinear} when the following diagram commutes:
      \begin{equation}\label[diagram]{diag:bilinearity}
        \DiagramSquare{
          width = 2.5cm,
          nw = LA\times LB,
          ne = L\prn{A\times B},
          se = C,
          sw = A\times B,
          north = \kappa_{A,B},
          west = \alpha_A\times\alpha_B,
          east = f^\dagger,
          south = f,
        }
      \end{equation}
    \end{definition}

    \begin{lemma}\label{lem:bistrict-iff-bilinear}
      A morphism $f\colon A\times B \to C$ is bistrict if and only if it is bilinear.
    \end{lemma}

    \begin{proof}
      A bilinear map is clearly bistrict. Conversely, assume that $f\colon A\times B\to C$ is bistrict. By \cref{cor:joint-epi}, both of the embeddings $\brk{\bot\mid\eta_A}\colon \One+A\hooktwoheadrightarrow LA$ and $\brk{\bot\mid\eta_B}\colon \One+B\hooktwoheadrightarrow LB$ are epimorphic, and therefore so is their cartesian product. Therefore, it suffices to
      consider the restriction of \cref{diag:bilinearity} from \cref{def:bilinear} along $\brk{\bot\mid\eta_A}\times\brk{\bot\mid \eta_B}\colon \prn{\One+A}\times\prn{\One+B}\hooktwoheadrightarrow LA\times LB$, or, equivalently, along each of the following four embeddings:
      \begin{align}
        \prn{\bot,\bot}     & \colon \One \hookrightarrow LA\times LB
        \\
        \prn{\eta_A,\bot}   & \colon A\hookrightarrow LA\times LB
        \\
        \prn{\bot,\eta_B}   & \colon B \hookrightarrow LA\times LB
        \\
        \prn{\eta_A,\eta_B} & \colon A\times B \hookrightarrow LA\times LB
      \end{align}

      From this reduction, it is easily seen that bistrictness implies bilinearity.\qed
    \end{proof}

    Now we recall Seal's construction of the tensor product.

    \begin{definition}[{Seal~\cite[\S2.2]{seal:2013}}]\label{def:seal:tensor}
      The \demph{tensor product} $A\boxtimes B$ of two pointed dcpos $A$ and $B$ is given by the following coequalier in $\Dcppo$, which exists by virtue of \cref{cor:lift-algebras-cocomplete,cor:lift-alg-characterisation}:
      \begin{equation*}
        \begin{tikzpicture}[diagram,baseline=(0.base)]
          \node(0) {$L\prn{LA\times LB}$};
          \node[right = 4cm of 0] (1) {$L\prn{A\times B}$};
          \node[right = 2.5cm of 1] (2) {$A\boxtimes B$};
          \draw[->,transform canvas={yshift=.1cm}] (0) to node[above] {$\kappa_{A,B}^\dagger$} (1);
          \draw[->,transform canvas={yshift=-.1cm}] (0) to node[below] {$L\prn{\alpha_A\times \alpha_B}$} (1);
          \draw[->>,exists] (1) to node[above] {$q_{A,B}$} (2);
        \end{tikzpicture}
      \end{equation*}
    \end{definition}

    Seal~\cite{seal:2013} proves a universal property for the tensor product with respect to bilinear morphisms.

    \begin{theorem}[{Seal~\cite{seal:2013}}]\label{thm:seal:tensor-represents}
      The tensor product $A\boxtimes B$ represents bilinear maps in the sense that for any bilinear morphism $f\colon A\times B\to C$ there exists a unique linear morphism $\bar{f}\colon A\boxtimes B\to C$ making the following triangle:
      \begin{equation*}
        \begin{tikzpicture}[diagram, baseline=(current bounding box.center)]
          \node (nw) {$L\prn{A\times B}$};
          \node[right = of nw] (ne) {$C$};
          \node[below = of nw] (sw) {$A\boxtimes B$};
          \draw[->] (nw) to node[left] {$q_{A,B}$} (sw);
          \draw[->] (nw) to node[above] {$f^\dagger$} (ne);
          \draw[->,exists] (sw) to node[sloped,below] {$\exists! \bar{f}$} (ne);
        \end{tikzpicture}
      \end{equation*}

      Moreover, let ${\boxtimes_{A,B}}\colon A\times B\to A\boxtimes B$ be the composite \[A\times B\xrightarrow{\eta_{A\times B}}L\prn{A\times B}\xrightarrow{q_{A,B}}A\boxtimes B\text{.}\] Then for any linear morphism $h\colon A\boxtimes B \to C$, the restriction $h\circ{\boxtimes_{A,B}} \colon A\times B \to C$ is bilinear and induces $h$ in the sense that $\overline{h\circ{\boxtimes_{A,B}}} = h$.
    \end{theorem}

    \begin{proof}
      This follows from Seal~\cite[Lemma~2.3.3]{seal:2013} via \cref{lem:lifting-strong}.\qed
    \end{proof}

    In order to show that Seal's tensor product satisfies the same universal property as our smash product, we must deduce a slight reformulation of \cref{thm:seal:tensor-represents}.

    \begin{lemma}[Universal bilinear map]\label{lem:universal-bilinear-map}
      The composite  \[\boxtimes_{A,B} = A\times B\xrightarrow{\eta_{A\times B}}L\prn{A\times B}\xrightarrow{q_{A,B}}A\boxtimes B\]
      is the \demph{universal bilinear map} in the sense that any bilinear map $f\colon A\times B\to C$ factors uniquely through it in $\Dcppo$ as depicted below:
      \begin{equation*}
        \begin{tikzpicture}[diagram,baseline=(current bounding box.center)]
          \node (nw) {$A\times B$};
          \node[right = of nw] (ne) {$C$};
          \node[below = of nw] (sw) {$A\boxtimes B$};
          \draw[->] (nw) to node[left] {$\boxtimes_{A,B}$} (sw);
          \draw[->] (nw) to node[above] {$f$} (ne);
          \draw[->,exists] (sw) to node[sloped,below] {$\exists! \bar{f}$} (ne);
        \end{tikzpicture}
      \end{equation*}
    \end{lemma}

    \begin{proof}
      Notice that ${\boxtimes_{A,B}}\colon A\times B\to A\otimes B$ is indeed bilinear by the second part of \cref{thm:seal:tensor-represents}. That any bilinear map $f\colon A\times B\to C$ factors uniquely through it follows from the first part of \cref{thm:seal:tensor-represents} via \cref{cor:joint-epi}. Indeed, we first let $\bar{f} \colon A\boxtimes B\to C$ be the map determined by \cref{thm:seal:tensor-represents} as follows:
      \begin{equation*}
        \begin{tikzpicture}[diagram, baseline=(current bounding box.center)]
          \node (nw) {$L\prn{A\times B}$};
          \node[right = of nw] (ne) {$C$};
          \node[below = of nw] (sw) {$A\boxtimes B$};
          \draw[->] (nw) to node[left] {$q_{A,B}$} (sw);
          \draw[->] (nw) to node[above] {$f^\dagger$} (ne);
          \draw[->,exists] (sw) to node[sloped,below] {$\bar{f}$} (ne);
        \end{tikzpicture}
      \end{equation*}

      By \cref{cor:joint-epi}, the diagram above commutes if and only if its restrictions along $\bot\colon \One\to L\prn{A\times B}$ and $\eta_{A\times B}\colon A\times B$ commute. The former is automatic because all maps in sight are strict; the latter is precisely the property of $\bar{f}$ extending $f$ along $\boxtimes_{A,B}$.\qed
    \end{proof}

    \begin{corollary}\label{cor:smash=tensor}
      There exists a unique bilinear / bistrict isomorphism $A\boxtimes B\to A\otimes B$ from Seal's tensor product to our smash product factoring the universal bistrict map through the universal bilinear map, and vice versa:
      \begin{equation*}
        \begin{tikzpicture}[diagram,baseline=(current bounding box.center)]
          \node (nw) {$A\times B$};
          \node[right = of nw] (ne) {$A\otimes B$};
          \node[below = of nw] (sw) {$A\boxtimes B$};
          \draw[->] (nw) to node[left] {$\boxtimes_{A,B}$} (sw);
          \draw[->] (nw) to node[above] {$\otimes_{A,B}$} (ne);
          \draw[->,exists] (sw) to node[sloped,below] {$\overline{\otimes_{A,B}}$} (ne);
        \end{tikzpicture}
        \quad
        \begin{tikzpicture}[diagram,baseline=(current bounding box.center)]
          \node (nw) {$A\times B$};
          \node[right = of nw] (ne) {$A\boxtimes B$};
          \node[below = of nw] (sw) {$A\otimes B$};
          \draw[->] (nw) to node[left] {$\otimes_{A,B}$} (sw);
          \draw[->] (nw) to node[above] {$\boxtimes_{A,B}$} (ne);
          \draw[->,exists] (sw) to node[sloped,below] {$\overline{\boxtimes_{A,B}}$} (ne);
        \end{tikzpicture}
      \end{equation*}
    \end{corollary}

    \begin{proof}
      This is an immediate consequence of the fact that bilinear and bistrict maps coincide (\cref{lem:bistrict-iff-bilinear}).\qed
    \end{proof}

  \end{xsect}

  \begin{xsect}{Symmetric monoidal structure of the smash product}
    The smash product of pointed dcpos from \cref{sec:smash-product} extends to a full symmetric monoidal structure on $\LiftAlg = \Dcppo = \Ipo$ with identity $I = L\One$; this result can be taken off the shelf from Seal~\cite[Theorem~2.5.5]{seal:2013}, in combination with our own result $A\otimes B = A\boxtimes B$ from \cref{cor:smash=tensor}.
  \end{xsect}

  \begin{xsect}{Symmetric monoidal structure of the lifting adjunction}
    Seal~\cite{seal:2013} shows that under assumptions that we have established in this paper for the lifting monad $\LL$ and its category of algebras $\LiftAlg=\Dcppo=\Ipo$, the Eilenberg--Moore adjunction $L\dashv U\colon \Dcppo\to\Dcpo$ is \emph{monoidal}: the left adjoint is \emph{strong monoidal} (\cf our own \cref{lem:commutator-is-universal}) and the right adjoint is \emph{lax monoidal}.

    In this section, we extend the result of \opcit in our specific case to show that $L\dashv U\colon\Dcppo\to\Dcpo$ is \emph{symmetric} monoidal.
    We first recall the braiding $\beta^\otimes_{A,B}\colon A\otimes B \to B\otimes A$ of the smash product in $\Dcppo$ in terms of the braiding of the Cartesian product on $\Dcpo$:
    \begin{equation}\label[diagram]{diag:smash-braiding}
      \begin{tikzpicture}[diagram, baseline=(current bounding box.center)]
        \node (nw) {$A\times B$};
        \node[right = of nw] (ne) {$B\times A$};
        \node[right = 2.25cm of ne] (nee) {$B\otimes A$};
        \node[below = of nw] (sw) {$A\otimes B$};
        \draw[->] (nw) to node[left] {$\otimes_{A,B}$} (sw);
        \draw[->] (nw) to node[above] {$\beta^\times_{A,B}$} (ne);
        \draw[->] (ne) to node[above] {$\otimes_{B,A}$} (nee);
        \draw[->,exists] (sw) to node[sloped,below] {$\exists! \beta^\otimes_{A,B}$} (nee);
      \end{tikzpicture}
    \end{equation}

    \begin{lemma}\label{lem:L-symmetric}
      The functor $L\colon \Dcpo\to \Dcppo$ is symmetric monoidal in the sense that the following diagram commutes in $\Dcppo$ for dcpos $A,B,C$:
      \begin{equation}\label[diagram]{diag:L-symmetric}
        \DiagramSquare{
          nw = LA\otimes LB,
          sw = L\prn{A\times B},
          ne = LB\otimes LA,
          se = L\prn{B\times A},
          south = L\prn{\beta^\times_{A,B}},
          north = \beta^\otimes_{A,B},
          west = \bar\kappa_{A,B},
          east = \bar\kappa_{B,A},
          width = 3cm,
        }
      \end{equation}
    \end{lemma}

    \begin{proof}
      To check that \cref{diag:L-symmetric} commutes, it suffices to consider its restriction along the universal bistrict map $\otimes_{LA,LB}\colon LA\times LB\to LA\otimes LB$. Therefore, to check that the lower inner square commutes in \cref{diag:L-symmetric:2} below, it suffices to check that the outer square commutes in the sense that $L\prn{\beta^\times_{A,B}}\circ \kappa_{A,B} = \kappa_{B,A}\circ \beta^\times_{LA,LB}$:
      \begin{equation}\label[diagram]{diag:L-symmetric:2}
        \begin{tikzpicture}[diagram,baseline=(current bounding box.center)]
          \SpliceDiagramSquare<sq/>{
            nw = LA\otimes LB,
            sw = L\prn{A\times B},
            ne = LB\otimes LA,
            se = L\prn{B\times A},
            south = L\prn{\beta^\times_{A,B}},
            north = \beta^\otimes_{A,B},
            west = \bar\kappa_{A,B},
            east = \bar\kappa_{B,A},
            north/node/style = upright desc,
            west/node/style = upright desc,
            east/node/style = upright desc,
            width = 3.25cm,
          }
          \node[between = sq/sw and sq/ne] {$?$};
          \node[above = of sq/nw] (nw) {$LA\times LB$};
          \node[above = of sq/ne] (ne) {$LB\times LA$};
          \draw[->] (nw) to node[upright desc] {$\otimes_{LA,LB}$} (sq/nw);
          \draw[->] (nw) to node[above] {$\beta^\times_{LA,LB}$} (ne);
          \draw[->] (ne) to node[upright desc] {$\otimes_{LB,LA}$} (sq/ne);
          \draw[->,bend left=50] (ne) to node[right] {$\kappa_{B,A}$} (sq/se);
          \draw[->,bend right=50] (nw) to node[left] {$\kappa_{A,B}$} (sq/sw);
        \end{tikzpicture}
      \end{equation}

      By \cref{cor:joint-epi} and the fact that all maps in sight are strict, it suffices to consider just three cases:
      \begin{align*}
        L\prn{\beta^\times_{A,B}}\prn{\kappa_{A,B}\prn{\eta_A x,\bot}}
         & =
        L\prn{\beta^\times_{A,B}}\bot
        \\
         & = \bot
        \\
         & =
        \kappa_{B,A}\prn{\bot,\eta_A x}
        \\
         & = \kappa_{B,A}\prn{\beta^\times_{LA,LB}\prn{\eta_A x,\bot}}
        \\
        L\prn{\beta^\times_{A,B}}\prn{\kappa_{A,B}\prn{\bot,\eta_B y}}
         & =
        L\prn{\beta^\times_{A,B}}\bot
        \\
         & = \bot
        \\
         & =
        \kappa_{B,A}\prn{\eta_B y,\bot}
        \\
         & =
        \kappa_{B,A}\prn{\beta^\times_{LA,LB}\prn{\bot,\eta_B y}}
        \\
        L\prn{\beta^\times_{A,B}}\prn{\kappa_{A,B}\prn{\eta_Ax,\eta_by}}
         & =
        L\prn{\beta^\times_{A,B}}\prn{\eta_{A\times B}\prn{x,y}}
        \\
         & =
        \eta_{B\times A}\prn{\beta^\times_{A,B}\prn{x,y}}
        \\
         & =
        \eta_{B\times A}\prn{y,x}
        \\
         & =
        \kappa_{B,A}\prn{\eta_By,\eta_Ax}
        \\
         & = \kappa_{B,A}\prn{\beta^\times_{LA,LB}\prn{\eta_Ax,\eta_By}}
      \end{align*}

      We are done.\qed
    \end{proof}

    \begin{lemma}\label{lem:U-symmetric}
      The forgetful functor $U\colon \Dcppo\to \Dcpo$ is symmetric monoidal in the sense that the following diagram commutes in $\Dcpo$ for pointed dcpos $A,B,C$:
      \begin{equation}\label[diagram]{diag:U-symmetric}
        \DiagramSquare{
          nw = A\times B,
          ne = B\times A,
          north = \beta^\times_{A,B},
          sw = A\otimes B,
          se = B\otimes A,
          west = \otimes_{A,B},
          east = \otimes_{B,A},
          south = \beta^\otimes_{A,B},
          width = 2.5cm,
        }
      \end{equation}

      \begin{proof}
        That \cref{diag:U-symmetric} commutes is in fact the \emph{defining} property of the braiding $\beta^\otimes_{A,B}$ as constructed in \cref{diag:smash-braiding}.\qed
      \end{proof}
    \end{lemma}

    \begin{corollary}\label{cor:lifting-adjunction-symmetric-monoidal}
      The adjunction $L\dashv U\colon \Dcppo\to \Dcpo$ is symmetric monoidal in the sense that $L\colon \Dcpo\to\Dcppo$ is \emph{strong} symmetric monoidal and $U\colon\Dcppo\to\Dcpo$ is \emph{lax} symmetric monoidal.
    \end{corollary}

    \begin{proof}
      By \cref{lem:L-symmetric,lem:U-symmetric} via Seal~\cite[Remark~2.7.3]{seal:2013}.\qed
    \end{proof}
  \end{xsect}

  \begin{xsect}{Closed structure of the lifting adjunction}
    Kock~\cite{kock:1971} has provided a method to lift the closed structure of $\Dcpo$ to $\LiftAlg$ by means of an equaliser of dcpos. Of course, the forgetful functor $U\colon\Dcppo\to\Dcpo$ is monadic (\cref{cor:lift-alg-characterisation}) and so creates limits; therefore we can slightly reformulate the construction of \opcit by computing an equaliser of pointed dcpos directly.

    \begin{definition}
      Let $A$ and $B$ be pointed dcpos. We define the \demph{linear function space} $A\multimap B$ to be the following equaliser in $\Dcppo$, where $\sigma_{A,B}\colon B^A\to B^{LA}$ is the internal extension map induced by the strength of $L$ and the algebra structure of $B$:
      \begin{equation*}
        \begin{tikzpicture}[diagram,baseline=(0.base)]
          \node (0) {$A\multimap B$};
          \node[right = of 0] (1) {$B^A$};
          \node[right = of 1] (2) {$B^{LA}$};
          \draw[exists,>->] (0) to (1);
          \draw[->, transform canvas={yshift=.15cm}] (1) to node[above] {$B^{\alpha_A}$} (2);
          \draw[->,transform canvas={yshift=-.15cm}] (1) to node[below] {$\sigma_{A,B}$} (2);
          \useasboundingbox ($(0)+(-1,-1)$) rectangle ($(2)+(1,1)$);
        \end{tikzpicture}
      \end{equation*}
    \end{definition}

    The results of Kock~\cite{kock:1971} then imply that the adjunction $L\dashv U\colon\Dcppo\to \Dcpo$ is closed with respect to the linear function space.

    \begin{definition}
      Let $A$ and $B$ be pointed dcpos. We define the \demph{strict function space} $A\Rightarrow_\bot B$ to be the following equaliser in $\Dcppo$:
      \begin{equation}\label[diagram]{diag:strict-function-space}
        \begin{tikzpicture}[diagram,baseline=(0.base)]
          \node (0) {$A\Rightarrow_\bot B$};
          \node[right = of 0] (1) {$B^A$};
          \node[right = of 1] (2) {$B$};
          \draw[exists,>->] (0) to (1);
          \draw[->, transform canvas={yshift=.15cm}] (1) to node[above] {$B^{\bot}$} (2);
          \draw[->,transform canvas={yshift=-.15cm}] (1) to node[below] {$\bot\circ !_{B^A}$} (2);
          \useasboundingbox ($(0)+(-1,-1)$) rectangle ($(2)+(1,1)$);
        \end{tikzpicture}
      \end{equation}
    \end{definition}

    \begin{lemma}\label{lem:strict-vs-linear-closed-structure}
      The strict and linear function spaces coincide.
    \end{lemma}

    \begin{proof}
      We will show that for any strict map $f\colon C\to B^A$, we have $B^\bot\circ f = \bot\circ {!}_{B^A}\circ f$ if and only if $B^{\alpha_A}\circ f = \sigma_{A,B}\circ f$.
      Fixing $x:C$, we must check that $fx \bot = \bot$ if and only if $f\,x\,\circ \alpha_A  = \sigma_{A,B}\circ {fx}$. These are equivalent by \cref{cor:joint-epi} and the unit laws for algebras.\qed
    \end{proof}

    By virtue of \cref{lem:strict-vs-linear-closed-structure}, we will freely write $A\multimap B$ for both the linear and strict function spaces.

    \begin{lemma}\label{lem:smash-lolli}
      For any pointed dcpo $A$, we have an adjunction $-\otimes A \dashv A\multimap -$ on $\Dcppo$.
    \end{lemma}

    \begin{proof}
      Fix $\bar{f}\colon C\otimes A\to B$ for some bistrict $f\colon C\times A\to B$. By definition, the mate $f^\sharp\colon C\to B^A$ in $-\times A\dashv \prn{-}^A$ is strict and moreover satisfies the defining property of \cref{diag:strict-function-space}, so we may factor $f^\sharp\colon C\to B^A$ through $A\multimap B\rightarrowtail B^A$ by some unique strict map $\bar{f}^\sharp\colon C\to A\multimap B$. It can be seen that this assignment is naturally bijective.\qed
    \end{proof}

  \end{xsect}

\end{xsect}

\section*{Acknowledgments}

I am thankful to Mart\'in Escard\'o, Marcelo Fiore, Jean Goubault-Larrecq, Daniel Gratzer, Ohad Kammar, and Tom de Jong for helpful discussions and consultations. I am grateful to Ralf Hinze and Dan Marsden for sharing their method for typesetting string diagrams. I especially owe the anonymous referee a debt of gratitude for their careful reading of the paper and many constructive criticisms and suggestions.

This work was funded in part by the European Union under the Marie Sk\l{}odowska-Curie Actions Postdoctoral Fellowship project \emph{TypeSynth: synthetic methods in program verification}, and in part by the United States Air Force Office of Scientific Research under grant FA9550-23-1-0728 (\emph{New Spaces for Denotational Semantics}; Dr.\ Tristan Nguyen, Program Manager). Views and opinions expressed are however those of the authors only and do not necessarily reflect those of the European Union, the European Commission, nor AFOSR. Neither the European Union nor the granting authority can be held responsible for them.

\bibliographystyle{splncs04}
\bibliography{refs-bibtex}

\end{document}